\documentclass[a4paper,twoside,11pt]{article}

\usepackage{BeaCed_article}

\title{Height representation of XOR-Ising loops\\
via bipartite dimers}
\author{C\'edric Boutillier \& B\'eatrice de Tili\`ere
\thanks{{\small
Laboratoire de Probabilit\'es et Mod\`eles Al\'eatoires, UMR~7599, Universit\'e Pierre et Marie Curie, 4 place Jussieu, 
F-75005 Paris.}
{\small\texttt{cedric.boutillier@upmc.fr},\small\texttt{
beatrice.de\_tiliere@upmc.fr}}
}}
\date{}

\begin{document}

\maketitle

\begin{abstract}
The XOR-Ising model on a graph consists of random spin configurations
on vertices of the graph obtained by
taking the product at each vertex of the spins of two
independent
Ising models. In this paper, we explicitly
relate loop configurations of the XOR-Ising model and those of a dimer model
living on a decorated, bipartite version of the Ising graph. This result is
proved for graphs embedded in compact surfaces of genus $g$.
Using this fact, we then prove that XOR-Ising loops have the same law as level
lines of the height function of this bipartite dimer model. At criticality, the
height function is known to converge weakly in distribution to
$\frac{1}{\sqrt{\pi}}$ a Gaussian free field \cite{Bea2}. As a consequence,
results of this paper shed a light on the occurrence of the Gaussian free field
in the XOR-Ising model. 
In particular, they prove a discrete analogue of Wilson's conjecture \cite{WilsonXOR}, stating that the scaling
limit of XOR-Ising loops are ``contour lines'' of the Gaussian free field.
\end{abstract}

\section{Introduction}

The \emph{double Ising model} consists of two Ising models, living on the same
graph. It is related \cite{KadanoffWegner,Wu71,Fan72b,Wegner}
to other models of statistical mechanics, as the 8-vertex
model \cite{sutherland,FanWu} and the Ashkin--Teller model \cite{AshkinTeller}.
In general, the two models
may be interacting. However, in this paper, we consider the case of
two non-interacting Ising models, defined on
the dual
$G^*=(V^*,E^*)$ of a graph $G=(V,E)$, having the same coupling constants
$(J_{e^*})_{e^*\in E^*}$, where the graph $G$ is
embedded either in a compact, orientable, boundaryless surface $\Sigma$ of genus $g\geq 0$, or
in the plane.

We are interested in the polarization of the model \cite{KadanoffBrown}, also
referred to as the XOR-\emph{Ising model} \cite{WilsonXOR} by Wilson. It is
defined
as follows: given a pair of spin
configurations $(\sigma,\sigma')\in\{-1,1\}^{V^*}\times
\{-1,1\}^{V^*}$,
the XOR-\emph{spin configuration} belongs to $\{-1,1\}^{V^*}$ and
is obtained by taking, at every vertex, the product of the spins.
The interface between $\pm 1$ spin
configurations of the XOR-configuration is a loop configuration of the
graph~$G$. Using extensive simulations, Wilson
\cite{WilsonXOR} finds that, when $G$ is a specific simply connected domain of
the plane, and when both Ising models are
critical, XOR loop configurations seem to have the same limiting behavior 
as ``contour lines'' of the Gaussian free field,
with heights of the contours spaced $\sqrt{2}$ times as far apart as they should be 
for the double dimer model on the square lattice. Similar conjectures involving
SLE rather than the Gaussian free field, are obtained through conformal field
theory \cite{IkhlefRajabpour,Picco}. Results of this paper explain the occurrence of the Gaussian free field in the 
XOR-Ising model and prove a discrete analogue of Wilson's conjecture.

The first part of this paper concentrates on finite graphs embedded in surfaces.
We explicitly relate XOR loop configurations to loop configurations in a
bipartite dimer model, implying in particular that
 both loop configurations have the same probability
distribution. In the second part, we prove that this correspondence still holds
for a large class of infinite planar graphs, the
so-called \emph{isoradial graphs} \cite{Kenyon3,KeSchlenk}, at criticality,
and make the connection with Wilson's conjecture. Here is an outline.

\begin{center}
\textbf{Outline}
\end{center}

\textbf{Section~\ref{sec1}}. One of the tools required is 
a version of Kramers and Wannier's low/high-temperature duality
\cite{KramersWannier1,KramersWannier2} in the case
of graphs embedded in surfaces of genus $g$, \emph{with boundary}. In the
literature, we did find versions of
this duality for graphs embedded in surfaces of genus $g$ \cite{LiGuo}, but we
could not find versions taking into account boundaries. This is the subject of
Propositions~\ref{prop:lowtemp} and~\ref{prop:hightemp}, it involves relative
homology theory and the Poincar\'e--Lefschetz duality.

\textbf{Sections~\ref{sec:2} and~\ref{sec:mixedcontour}} consist of the
extension to general graphs embedded in a surface of genus~$g$ of an expansion
due to Nienhuis \cite{Nienhuis}, which can be summarized as follows. Consider
the low-temperature expansion of
the double Ising model, \emph{i.e.}, consider pairs of polygon configurations
separating
clusters of $\pm 1$ spins of
each spin configuration. Drawing both polygon configurations on $G$ yields an
edge configuration consisting of \emph{monochromatic edges}, that is edges
covered by exactly one of the two polygon configurations, and \emph{bichromatic
edges}, that is edges covered by both polygon configurations. Monochromatic edge
configurations exactly correspond to XOR loop configurations, and separate
the surface $\Sigma$ into connected components $\Sigma_1,\ldots,\Sigma_N$.
Inside each connected component, the law of bichromatic edge
configurations is that of the low-temperature expansion of an Ising model with
coupling constants that are doubled. As a consequence, the partition
function of the double Ising model can be rewritten using XOR loop
configurations and bichromatic edge configurations, see Proposition~\ref{cor:Zisingcarree}.

Fixing a monochromatic edge configuration, and applying
low/high-temperature duality to the single Ising model corresponding to
bichromatic edges, yields a rewriting of the double Ising
partition function, as a sum over
pairs of non-intersecting polygon configurations of the primal and dual
graph, where primal polygon configurations exactly correspond to XOR
loop configurations, see Proposition~\ref{prop:mixedcontour2} and
Corollary~\ref{prop:mixedcontour}. Note that there are quite a few difficulties
in the proofs, due to the fact that we work on a
surface of genus $g$.

\begin{prop}\label{thm:main3}$\,$\\
The double Ising partition function for a graph embedded on a surface of genus
$g$ can be rewritten as:
\begin{equation*}
\Zdising(G^*,J)=\C_{\mathrm{I}}\sum_{
\substack{\{(P,P^*)\in\P^0(G)\times\P^0(G^*):\\
P\cap P^*=\emptyset\} }
}
\left(\prod_{e\in P}\frac{2e^{-2J_{e^*}}}{1+e^{-4J_{e^*}}}\right)
\left(\prod_{e^*\in P^*}\frac{1-e^{-4J_{e^*}}}{1+e^{-4J_{e^*}}}\right),
\end{equation*}
where primal polygon configurations of $\P^0(G)$ are the $\XOR$ loop
configurations, and $\C_{\mathrm{I}}=2^{|V^*|+2g+1}\left(\prod_{e\in E}\cosh
(2J_{e^*})\right)$.
\end{prop}
 
\textbf{Section~\ref{sec5}}. In Section~\ref{sec31}, we define the
6-vertex model on the medial graph $\GM$ constructed from $G$. Reformulating
an argument of Nienhuis \cite{Nienhuis}, we prove that the 6-vertex partition
function can be written as a sum over non-intersecting pairs of 
polygon configurations of the primal and dual graphs. 

In Section~\ref{sec:52}, we define the dimer model on the decorated, bipartite
graph $\GQ$
constructed from~$G$. Then, we present the
mapping between dimer configurations of $\GQ$ and \emph{free-fermionic} 6-vertex configurations of
$\GM$ \cite{WuLin,Dubedat}.
Using both mappings, one assigns to every dimer
configuration $M$ a pair $\poly(M)=(\poly_1(M),\poly_2(M))$ of non-intersecting
primal and dual polygon configurations. 
The weights of the 6-vertex model chosen to match those of edges in the mixed
contour expansion of the double Ising model satisfy the \emph{free-fermionic}
condition. As a consequence, we then obtain, see also
Proposition~\ref{prop:quadritilings}:
\begin{prop}\label{prop:main}
The dimer model partition function $\Zquadri^0(\GQ,J)$ can be rewritten as:
$$
\Zquadri^0(\GQ,J)=2\sum_{
\substack{\{(P,P^*)\in\P^0(G)\times\P^0(G^*):\\
P\cap P^*=\emptyset\} }
}
\left(\prod_{e\in P}\frac{2e^{-2J_{e^*}}}{1+e^{-4J_{e^*}}}\right)
\left(\prod_{e^*\in P^*}\frac{1-e^{-4J_{e^*}}}{1+e^{-4J_{e^*}}}\right),
$$
where primal and dual polygon configurations of $\P^0(G)\times \P^0(G^*)$
are the $\poly$ configurations.
\end{prop}

Combining Proposition~\ref{thm:main3} and Proposition~\ref{prop:main}
yields the following, see also Theorem~\ref{thm:main1inside}:
\begin{thm}\label{thm:intro-equal-law}
XOR loop configurations of the double Ising model on $G^*$ have the
same law as $\poly_1$ configurations of the corresponding dimer model
on the bipartite graph $\GQ$:
$$
\forall P\in\P^{0}(G),\quad\PPdising[\XOR=P]=\PPquadri^0[\poly_1=P].
$$
\end{thm}

\begin{rem}$\,$
In the paper \cite{Dubedat}, Dub\'edat relates a version of the double Ising
model and
the same bipartite dimer model in two ways. The first approach uses explicit
mappings, most of which are present in the physics literature, and
goes as follows. Consider a slightly different version of the double Ising
model, with one model living on the primal graph $G$ and the second
on the dual graph $G^*$. This double Ising model can be mapped to an
8-vertex model \cite{KadanoffWegner,Wu71} on the medial graph. Using Fan and
Wu's
abelian duality, this
8-vertex model \cite{FanWu} can be mapped to a second 8-vertex on the same
graph.
When coupling constants of the two Ising models satisfy Kramers and Wannier's
duality,
the second 8-vertex model is in fact a free-fermionic
6-vertex model. The free-fermionic 6-vertex model can in turn be mapped to a 
bipartite dimer model, a result due to \cite{WuLin} in the case of the square
lattice, and extended by \cite{Dubedat} in the general lattice case.
It can also be seen as a specific case of the mapping of the
free-fermionic 8-vertex model to a non-bipartite dimer model of \cite{FanWu}.
Note that this bipartite dimer model is the model of \emph{quadri-tilings}
studied by the second author in \cite{Bea1} and \cite{Bea2}.

When performing the different steps of the mapping, Dub\'edat keeps track of 
order/disorder variables, in the vein of
\cite{KadanoffCeva}. Using results of a previous paper of his
\cite{Dubedat:torsion}, this allows him to compute critical correlators in the
plane. For simply connected regions,
this result has independently been obtained by Chelkak, Hongler and Izyurov
\cite{ChelkakHonglerIzyurov}. 

Our goal here is different, since we aim at keeping track of
XOR-configurations. This information is not directly available in the above approach. Indeed Fan
and Wu's abelian duality for the 8-vertex model can be compared to a high-temperature expansion, 
where configurations cannot be interpreted using the
initial model. Note that expanding and recombining the identities of
\cite{Dubedat} involving correlators, one can recover 
the identity in law of Theorem~\ref{thm:intro-equal-law}; this proves the
existence of a coupling between the two models,
which we explicitly provide in this paper.

The second approach uses transformations on matrices. The partition
function of the double Ising model can be expressed using the determinant of the
Kasteleyn matrix of the Fisher graph \cite{Fisher}; whereas the partition
function of the bipartite dimer model can be expressed using the Kasteleyn
matrix of the graph $\GQ$. Dub\'edat shows that the two matrices are related
through transformations not affecting the determinant. Using the fact
that the partition function of the double Ising model is also related to the
determinant of the Kac--Ward matrix \cite{KacWard}, Cimasoni and Duminil-Copin
use the
same approach to relate the Kac--Ward matrix to the matrix of the same bipartite
dimer model \cite{CimasoniDuminil}. Their purpose is to identify the critical
point of general bi-periodic Ising models, see also Li
\cite{Li:spectral,Li:critical} for the case of the square lattice with arbitrary
fundamental domain.

However, the above transformations on matrices are not easily interpreted in terms of
transformations on configurations, and the relation to XOR-configurations
is not straightforward.

Using Nienhuis' mapping \cite{Nienhuis}, the main contribution of this paper
is to provide a coupling between the double Ising model and the bipartite dimer
model, which \emph{keeps track} of XOR loop configurations, and is valid for
graphs embedded in surfaces of genus $g$.
\end{rem}

\textbf{Section~\ref{sec:dIsing-iso}}.
    Suppose now that the two Ising models are critical and defined on the dual
of an
    infinite isoradial graph $G$ filling the whole plane, see
    Section~\ref{subsec:isoradial} for definitions. Then, the dimer model on the
    corresponding graph $\GQ$ is also critical in the dimer sense. 
    Using the locality property of both probability measures on
    Ising~\cite{BoutillierdeTiliere:iso_gen}, and dimer configurations~\cite{Bea2}
    on isoradial graphs at criticality, we prove that the equality in
    law stated in Theorem~\ref{thm:intro-equal-law} still holds in this infinite
    context. See Theorem~\ref{thm:maininfinite}.

\textbf{Section~\ref{sec:height}}. 
    The graph $\GQ$ being bipartite, using a height function denoted $h$, dimer
    configurations can naturally be interpreted as discrete random interfaces. Our
    second theorem, see also Theorem~\ref{thm:xor-level}, proves the following
    \begin{thm}\label{thm:intro-xor-level}
      XOR loop configurations of the double Ising model defined on $G^*$ have the same
      law as level lines of the restriction of the height function $h$ to vertices of
      the dual graph $G^*$.
    \end{thm}
    
    Theorem~\ref{thm:intro-xor-level} can be interpreted as a proof of Wilson's
    conjecture mentioned above (see~\ref{sec:wilsonconj} for a precise
    statement) in the discrete setting, since it is known that the height
    function~$h$, seen as a random distribution, converges in law in the scaling
    limit to $\frac{1}{\sqrt{\pi}}$ times the Gaussian free field
    in the plane \cite{Bea2}.
    In particular, we explain the special value of the spacing.
    It would yield a complete proof of the conjecture if we could overcome the same
    technical obstacles as those of the proof of the convergence of double dimer loops to
    $\text{CLE}_4$.

\emph{Acknowledgments}: We would like to warmly thank Thierry L\'evy for very
helpful discussions on relative homology. We are also grateful to both referees for their useful comments.

\section{Ising model on graphs embedded in surfaces}\label{sec1}

In this section, we let $G$ be a graph embedded in a compact,
orientable, boundaryless surface of genus $g$ ($g\geq 0$), and $G^*$ be
its dual graph. The embedding of $G^*$ is chosen such that dual vertices
are in the interior of the corresponding faces.

Fix some integer $p\geq 0$, and suppose first that $p\geq 1$. For every $i\in\{0,\ldots,p-1\}$, let
$B_i$ be a union of closed faces of $G$ homeomorphic to a disc, such that $\forall i\neq j$, 
$B_i\cap B_j=\emptyset$. Denote by $\Sigma$ the surface of genus $g$
from which the union of the interiors of $B_i$'s is removed. Then $\Sigma$ is a compact,
orientable surface of genus $g$, with boundary $\partial \Sigma=\partial
B_0\cup\cdots \partial B_{p-1}$. When $p=0$, then $\Sigma$ is the compact, orientable, boundaryless surface of genus $g$ in which 
the graph $G$ is embedded.

Let $G_{\Sigma}=(V_{\Sigma},E_{\Sigma})$ be the subgraph of $G$ defined as
follows:
$V_{\Sigma}$ consists of vertices of $V\cap\Sigma$; and
$E_{\Sigma}$ 
consists of edges of $E$ joining vertices of $V_{\Sigma}$, from which 
edges on the boundary $\partial \Sigma$ are removed. Let
$G_{\Sigma}^*=(V_{\Sigma}^*,E_{\Sigma}^*)$ be
the subgraph of $G^*$ whose vertices are vertices of $V^*\cap\Sigma$, and whose
edges are edges of $G^*$ joining vertices of $V_{\Sigma}^*$; see Figure
\ref{fig:ising_surface_1} for an example. Note that the
graph $G_{\Sigma}^*$ contains all edges dual to edges of $G_\Sigma$,
\emph{i.e.}, 
there is a bijection between primal edges of $G_\Sigma$ and dual edges of
$G_{\Sigma}^*$. Note that when $p=0$, $G_{\Sigma}=G$ and $G_{\Sigma}^*=G^*$.

\begin{figure}[ht]
\begin{center}
\includegraphics[width=10cm]{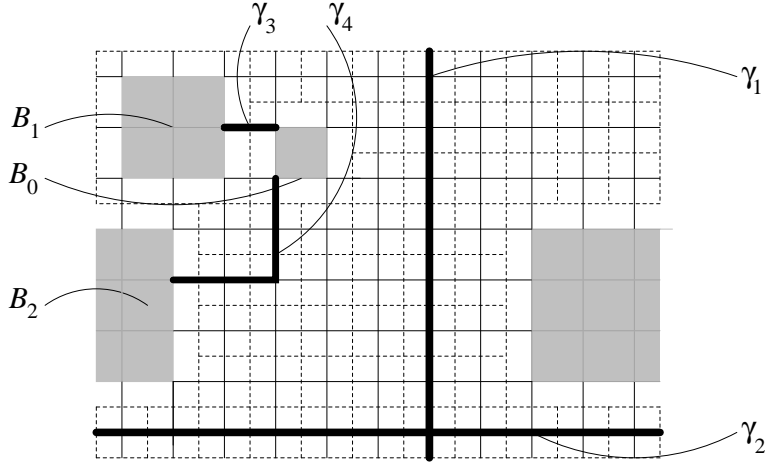} 
\caption{The graph $G$ is a piece of $\ZZ^2$
embedded in the torus. The union of faces $(B_i)_{i\in\{0,1,2\}}$ is pictured
in light grey. The graph $G_{\Sigma}$ consists of black plain lines,
and the dual graph $G_{\Sigma}^*$ of black dotted lines. The paths
$(\gamma_i)_{i=1}^4$ defining \emph{defects} of Section
\ref{sec11} are drawn in thick black lines.}\label{fig:ising_surface_1}
\end{center}
\end{figure}

Fix a collection of positive constants $(J_{e^*})_{e^*\in E^*}$ attached to
edges of $G^*$, referred to as \emph{coupling constants}. The \emph{Ising model
on $G_{\Sigma}^*$ with free boundary conditions and coupling constants $(J_{e^*})$} is defined as follows.
A \emph{spin configuration} $\sigma$ of $G_{\Sigma}^*$ is a function of the
vertices of
$V_{\Sigma}^*$ with values in $\{-1,+1\}$. The probability of occurrence of a
spin
configuration $\sigma$ is given by the \emph{Ising Boltzmann measure}, denoted
$\PPising$, and defined by:
\begin{equation*}
\forall\sigma\in \{-1,1\}^{V_{\Sigma}^*},\quad \PPising(\sigma) =
\frac{1}{\Zising(G_{\Sigma}^*,J)} \exp\left(
  \sum_{e^*=u^*v^*\in E_{\Sigma}^*}J_{e^*} \sigma_{u^*} \sigma_{v^*} \right),
\end{equation*}
where $\displaystyle
\Zising(G_{\Sigma}^*,J)=\sum_{\sigma\in\{-1,+1\}^{V_{\Sigma}^*}}\exp\left(
\sum_{e^*=u^*v^*\in E_{\Sigma}^*}J_{e^*} \sigma_{u^*} \sigma_{v^*}
\right)$ is the \emph{Ising partition function}. Note that
to simplify notation, the inverse temperature is included in the coupling
constants.

\subsection{Ising model with defect lines}\label{sec11}

When $p\geq 1$, let $N=2g+p-1$, and when $p=0$, let $N=2g$. We now define $2^N$ versions of the original Ising
model. Let $\rep{\gamma}_1,\cdots,\rep{\gamma}_N$ be $N$ unoriented paths consisting of
edges of the primal graph $G_{\Sigma}$; see Figure \ref{fig:ising_surface_1}
for an example, where 
\begin{itemize}
  \item for every $i\in\{1,\dots,g\}$, the paths ${\rep{\gamma}}_{2i-1},\rep{\gamma}_{2i}$
wind around the $i$-th handle in two transverse directions,
  \item when $p\geq 2$, for every $i\in\{1,\dots,p-1\}$, the path ${\rep \gamma}_{2g+i}$ joins
$\partial B_0$ and $\partial B_i$.
\end{itemize}
The paths $\gamma_i$'s are thought as sets of edges.
Let $\rep{\epsilon}$ be one
of the $2^N$ possible ``unions'' of paths
$\widehat{\cup}_{i\in I}\rep{\gamma}_i$, where
$I\subset\{1,\ldots,N\}$ and $\widehat{\cup}$ means that an edge with multiplicity $k$ is present 
iff $k\equiv 1 $ mod 2. Then, we change the sign of coupling
constants of dual edges intersecting with
$\rep{\epsilon}$. Spin configurations are defined as above, and so is
the probability measure on spin configurations. This defines the \emph{Ising
model with coupling constants $(J_{e^*})$ and defect condition $\rep{\epsilon}$.}

In fact, the appropriate framework for defining the Ising model with defects, is
relative homology theory, see Appendices \ref{app:A2}, \ref{app:A3}, and
\ref{app:graphs}. The \emph{first
homology group of $\Sigma$ relative to its boundary $\partial \Sigma $} is
denoted by $\relhomol{1}$. The collection of paths
$(\rep{\gamma}_1,\ldots,\rep{\gamma}_{N})$ defined above, is a representative of a basis
$\Gamma=(\gamma_1,\ldots,\gamma_{N})$ of the first relative homology group
$\relhomol{1}$ seen as a $\ZZ/2\ZZ$-vector space. In the case where $p=0$, $\partial \Sigma=\emptyset$ and 
the first homology group of $\Sigma$ relative to its boundary
 is simply the first homology group.

Let $\eps$ denote the relative homology class of $\rep{\epsilon}$ in $\relhomol{1}$.
Then, it will
be clear from the low-temperature expansion of Section~\ref{subsec:lowtemp} that
the partition function is independent of the choice of basis and of the choice
of representative of
$\eps$. As a consequence, we refer to this model as the 
\emph{Ising model with coupling constants $(J_{e^*})$ and defect condition
$\eps$}, and denote by
$\Zising^{\eps}(G_{\Sigma}^*,J)$ the corresponding partition function.
Nevertheless, since we want the change of signs of coupling constants to be well
defined throughout the paper, we fix representatives of relative homology
classes in $\relhomol{1}$, using the collection of paths
$\rep{\gamma}_1,\ldots,\rep{\gamma}_N$ defined above. Note that the original Ising model
introduced has defect
condition $\eps=0$ and $\rep{\epsilon}$ is empty.
Note also that this treatment is completely equivalent to considering the connected
components of the boundary as the boundary of marked faces, and allowing
insertion of disorder operators on these marked faces. However,
the formulation in terms of defect conditions is natural in our context:
the graphs on which the Ising model with defect conditions live, arise
from the surgery of a larger graph embedded in a surface, and as such, their
boundary have a real geometric meaning.

\subsection{Low- and high-temperature expansion}\label{subsec:lowtemp}

Proposition \ref{prop:lowtemp} below extends the \emph{low-temperature
expansion} of
Kramers and Wannier \cite{KramersWannier1,KramersWannier2} to the case of graphs
embedded
on a compact, orientable surface with boundary. It consists of rewriting the
Ising partition function as a sum over polygon configurations of the graph
$G_{\Sigma}$, ``separating''
clusters of $\pm 1$ spins; see Figure \ref{fig:low_temp} (left) for an example.

\begin{figure}[ht]
\begin{center}
\includegraphics[width=\linewidth]{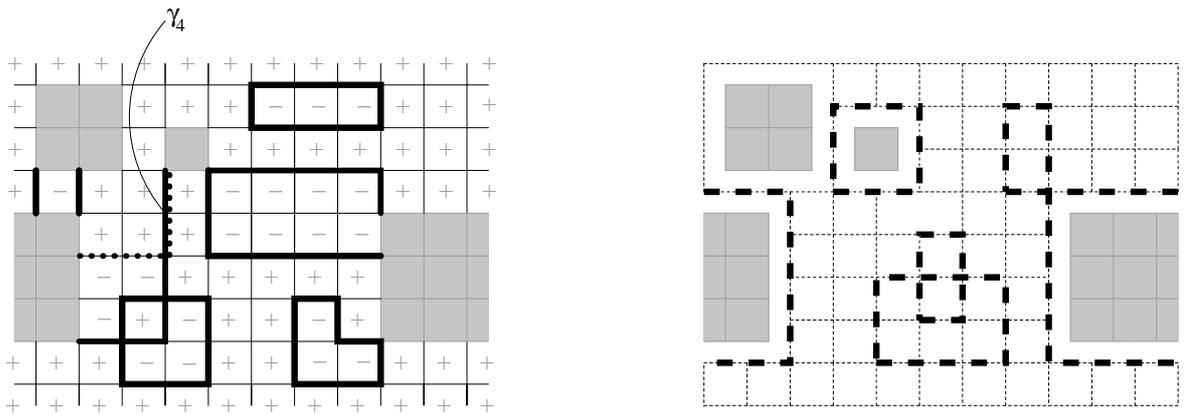} 
\caption{Left: polygon configuration of $G_{\Sigma}$ corresponding to a spin
configuration of the Ising model with defect condition $\underline{\epsilon}=\gamma_4$.
Right: polygon configuration of $G_{\Sigma}^*$.}\label{fig:low_temp}
\end{center}
\end{figure}

A \emph{polygon configuration} of $G_{\Sigma}$ is a subset of edges of
$G_\Sigma$, such
that vertices not on the boundary $\partial\Sigma$ are incident to an even
number of edges.
There is no restriction for vertices on the boundary $\partial \Sigma$. Let
us denote by
$\P(G_{\Sigma})$ the set of polygon configurations of $G_{\Sigma}$.

Let $\eps$ be an element
of $\relhomol{1}$, and let
$\P^{\eps}(G_{\Sigma})$ denote the set of polygon
configurations of $G_{\Sigma}$
whose relative homology class in $\relhomol{1}$ is $\eps$, meaning in particular that, for every $i$,
the number of edges
on $\partial B_i$ has the same parity as $\epsilon_{2g+i}$.

This
defines a partition of $\P(G_{\Sigma})$:
$$
\P(G_{\Sigma})=\bigcup_{\eps\in\relhomol{1}}\P^{\eps}(G_{\Sigma}).
$$

\begin{prop}[Low-temperature expansion]\label{prop:lowtemp}$\,$\\
For every relative homology class $\eps\in\relhomol{1}$,
\begin{equation}
\Zising^{\eps}(G_\Sigma^*,J)= 2\Bigl(\prod_{e\in E_{\Sigma}}
e^{J_{e^*}}\Bigr)
\sum_{P\in\P^{\eps}(G_{\Sigma})}
\prod_{e\in P}e^{-2J_e^*}.
\label{eq:lowtemp}
\end{equation} 
\end{prop}
\begin{proof}

Suppose for the moment that $\eps$ is the class $0 \in \relhomol{1}$, so
that we deal with the usual Ising model. Using the identity
\eqref{eq:idea_low_temp} below, one can rewrite the partition function as a
statistical sum over polygon configurations separating clusters of $\pm 1$
spins: if
$\sigma_{\dual u}$ and $\sigma_{\dual v}$ are
two neighboring spins of an edge $\dual e=\dual u\dual v$, then
\begin{equation}
  e^{J_{\dual e}\sigma_{\dual u} \sigma_{\dual v}} = e^{J_{\dual e}}\left(
\delta_{\{\sigma_{\dual u} =
  \sigma_{\dual v}\}} + e^{-2J_{\dual e}} \delta_{\{\sigma_{\dual u} \neq
  \sigma_{\dual v}\}} \right).
  \label{eq:idea_low_temp}
\end{equation}
When injecting the right hand side in the expression of the Ising partition
function, the product over dual edges $\dual e$ of $e^{J_{\dual e}}$ can
be factored out. Since primal and dual edges are in bijection, this can also be
written as a product over primal edges. Then, expanding the product, we get
a product of contributions
for all edges separating two neighboring spins with opposite signs. These
edges form a polygon configuration $P^0$ of $G_{\Sigma}$ separating clusters of
$\pm 1$ spins. As a consequence $P^0$ has homology class $0$, \emph{i.e.}, $P^0$
belongs to $\P^0(G_\Sigma)$. 

Conversely,
any polygon configuration of $\P^0(G_\Sigma)$ is the boundary of
exactly two spin configurations, one obtained from the other by negating all
spins, which explains the factor 2 on the right hand side of \eqref{eq:lowtemp}. 

Suppose now that $\eps \neq 0$. In the Ising model with defect condition $\eps$,
coupling constants of edges crossing paths of the representative $\rep{\epsilon}$ are
negated. For these edges, the relation
  \eqref{eq:idea_low_temp} should be replaced by the following:

  \begin{equation*}
    e^{-J_{e^*}\sigma_{u^*} \sigma_{v^*}} =
    e^{J_{e^*}}\left( \delta_{\{\sigma_{u^*}\neq \sigma_{v^*}\}}+
e^{-2J_{e^*}}\delta_{\{\sigma_{u^*}=\sigma_{v^*}\}}
    \right).
  \end{equation*}

  Note that, when comparing to
  \eqref{eq:idea_low_temp}, the two Kronecker symbols have been exchanged. As a
consequence, the construction of polygon configurations as
  above is slightly modified:  the
  edge configuration, denoted by $P$, constructed from a spin
configuration is obtained
  from $P^0$ by switching the state of every edge $e$ in
  $\rep{\epsilon}$; see Figure \ref{fig:low_temp} (left). Then, the relative homology
class of $P$ in $\relhomol{1}$ is:
  \begin{equation*}
    [P] = [P^0] + [\rep{\epsilon}] = 0 + \eps = \eps.
  \end{equation*}
As a consequence $P$ belongs to $\P^\eps(G_\Sigma)$ and this, independently of
the choice of representative of $\eps$. Conversely, any
element of $\P^\eps(G_\Sigma)$ is obtained twice in this way.
\end{proof}

For the sequel, it is useful to introduce a symbol for the sum over
polygon configurations of the low-temperature expansion. For
$\eps\in\relhomol{1}$,
define
  \begin{equation*}
    \ZLT^{\eps}(G_\Sigma, J) = \sum_{P\in\P^{\eps}(G_\Sigma)}\Bigl(\prod_{e\in
P} e^{-2J_{\dual e}}\Bigr).
  \end{equation*}
The partition function of the Ising model with
defect condition $\eps$ can thus be rewritten as:
\begin{equation*}
  \Zising^{\eps}(G_{\Sigma}^*,J) = 2 \Bigl(\prod_{e\in E_\Sigma} e^{J_{\dual
e}} \Bigr)
  \ZLT^{\eps}(G_{\Sigma},J).
\end{equation*}

Proposition \ref{prop:hightemp} below extends the \emph{high-temperature
expansion} \cite{KramersWannier1,KramersWannier2,Wannier}
to the case of graphs embedded in a
compact, orientable surface with boundary. It consists of rewriting the
Ising partition function as a sum over polygon configurations of the graph
$G_{\Sigma}^*$, this time. In this case, polygon configurations do not have a
simple interpretation in terms of spin configurations.

A \emph{polygon configuration} of $G_{\Sigma}^*$
(or simply \emph{dual polygon configuration})
is a subset of edges such
that each vertex of $G_{\Sigma}^*$ is incident to an even number of edges, see
Figure \ref{fig:low_temp} (right) for an example. It is
thus a union of closed cycles on $\dual G_\Sigma$. Let us denote
by $\P(G_{\Sigma}^*)$ the set of polygon configurations of $G_{\Sigma}^*$.

Let $\homol{1}$ be the first homology group of $\Sigma$, see Appendices
\ref{app:A1}, \ref{app:A3} and \ref{app:graphs}.
Then, to each dual polygon configuration is
assigned its homology class in $\homol{1}$. For every $\tau\in\homol{1}$, we let
$\P^{\tau}(G_{\Sigma}^*)$ denote the set of dual polygon
configurations restricted to having homology class $\tau$ in
$\homol{1}$. This defines a partition of $\P(G_{\Sigma}^*)$:

\begin{equation*}
  \P(G_{\Sigma}^*)=\bigcup_{\tau\in\homol{1}}\P^{\tau}(G_{\Sigma}^*).
\end{equation*}

\begin{prop}[High-temperature expansion]\label{prop:hightemp} $\,$\\
For every relative homology class $\eps\in\relhomol{1}$,
\begin{equation}\label{eq:high_temp}
\Zising^{\eps}(G_{\Sigma}^*,J)=
2^{|V_{\Sigma}^*|}\Bigl(\prod_{e\in E_{\Sigma}} \cosh(J_{e^*})\Bigr)
\cdot \sum_{\tau\in\homol{1}}\Bigl[
(-1)^{(\tau|\eps)}
\sum_{P^*\in\P^{\tau}(G_{\Sigma}^*)}
\Bigl(\prod_{e^*\in P^*}\tanh(J_{e^*})\Bigr)\Bigr],
\end{equation}
where $(\tau|\eps)$ is the \emph{intersection form} evaluated at $\tau$ and
$\eps$: it is the parity of the number of intersections of
any representative of $\tau$ and any representative
of $\eps$.
\end{prop}

For details on the intersection form, see Appendix \ref{app:intersection}. 

\begin{proof}
  This result is based on yet another way of rewriting the quantity $e^{\pm
J_{\dual  e}\sigma_{\dual u}\sigma_{\dual v}}$ for a dual edge
$e^*=u^*v^*$ of $E_{\Sigma}^*$. 
\begin{align}
    \label{eq:idea_high_temp}
    e^{\pm J_{e^*} \sigma_{u^*}\sigma_{v^*}}&= 
    \cosh J_{e^*}\pm \sigma_{u^*}\sigma_{v^*}\sinh J_{e^*}\nonumber\\ 
    &=\cosh J_{e^*} \left(1\pm \sigma_{u^*}\sigma_{v^*}\tanh J_{e^*} \right).
  \end{align}
The partition function is expanded into a sum of monomials in 
  $(\sigma_{u^*})_{u^*\in V_{\Sigma}^*}$. In the expansion, the spin variables
  come by pairs of neighbors $\sigma_{\dual u} \sigma_{\dual v}$ and thus can be
  formally identified with the dual edge connecting $\dual u$ and $\dual v$,
  associated with a weight $\pm\tanh J_{e^*}$.
  Each monomial is then interpreted as a subgraph of $G_{\Sigma}^*$,
  the degree of $\sigma_{u^*}$ being the degree of $u^*$ in the corresponding edge
  configuration.
  Because of the symmetry $\sigma\leftrightarrow -\sigma$, when re-summing over spin
  configurations $\sigma$, only terms having even degree in each variable
  remain, giving a factor 2 per dual vertex, and other contributions cancel.
  As a consequence, the contributing monomials
  correspond to even subgraphs \emph{i.e.}, polygon configurations of
  $\P(G_{\Sigma}^*)$.
 
  We now determine the sign of dual polygon configurations.
  Fix $\tau\in\homol{1}$ and a dual
  polygon configuration $P^*\in\P^{\tau}(G_{\Sigma}^*)$. Then,
 edges of $P^*$ carrying a negative weight are exactly those crossing
  edges of $\rep{\epsilon}$. As a consequence, the sign of the contribution of $P^*$
corresponds
to $(-1)$ to the parity of the number of edges of $P^*$ intersecting with
$\rep{\epsilon}$, this is exactly given by $(\tau|\eps)$. The dual polygon configuration
$P^*$ thus has sign $(-1)^{(\tau|\eps)}$.
\end{proof}

As in the case of the low-temperature expansion, it is useful to introduce a
notation for
the sum over dual polygon configurations of the high-temperature expansion. For
$\tau\in \homol{1}$, define:
\begin{equation*}
  \ZHT^{\tau}(G_\Sigma^*,J) = \sum_{P^*\in\P^{\tau}(G_{\Sigma}^*)}
\Bigl(\prod_{e^*\in P^*}\tanh(J_{e^*})\Bigr).
\end{equation*}

The relation between \eqref{eq:lowtemp} and \eqref{eq:high_temp} 
can then be rewritten in the following compact form. For every
relative homology class $\eps\in\relhomol{1}$:
\begin{equation}\label{eq:lowhighexp}
  \ZLT^{\eps}(G_{\Sigma},J) =   2^{|\dual{V}_{\Sigma}|-1}
  \left(\prod_{e\in E_\Sigma} \frac{\cosh(J_{\dual  e})}{e^{J_{\dual e}}}\right)
\sum_{\tau\in\homol{1}}\Bigl[(-1)^{(\tau|\eps)}\ZHT^{\tau}(G_{\Sigma}^*,J)\Bigr].
\end{equation}

\begin{rem}
Relation \eqref{eq:lowhighexp} can be inverted using the orthogonality identity:
\begin{equation}
  \sum_{\eps\in\relhomol{1}}(-1)^{(\tau|\eps)}
  (-1)^{(\tau'|\eps)}= 2^N \delta_{\tau,\tau'},
  \label{eq:orthogonality}
\end{equation}
where $N=2g+p-1$ when $p\geq 1$, and $N=2g$ when $p=0$. This orthogonality relation is proved as follows.
The summand can be rewritten as $(-1)^{(\tau-\tau'|\epsilon)}$.
The application $\epsilon\mapsto (-1)^{(\tau-\tau'|\epsilon)}$ is a group
homomorphism from $\relhomol{1}$ to $\ZZ/2\ZZ$.
When $\tau=\tau'$, this application is constant, equal to 1,
all terms in the sum \eqref{eq:orthogonality} equal 1, and the total sum equals~$2^N$.
Otherwise, since the intersection pairing is non-degenerate (see
Appendix~\ref{app:intersection}), the application
$\epsilon\mapsto (-1)^{(\tau-\tau'|\epsilon)}$
takes the values 1 and -1 the same number of times, and the sum
\eqref{eq:orthogonality} is zero.
Using this identity, we obtain the inverted version of relation \eqref{eq:lowhighexp}:
\begin{equation*}
  \ZHT^{\tau}(\dual G_\Sigma, J) = 2^{-N-|\dual{V}_\Sigma|+1} 
  \left(\prod_{e\in E_\Sigma} \frac{e^{J_{\dual e}}}{\cosh(J_{\dual e})}\right)
  \sum_{\epsilon\in\relhomol{1}} \Bigl[ (-1)^{(\tau|\eps)}
    \ZLT^{\epsilon}(G_{\Sigma},J)\Bigr].
\end{equation*}

\end{rem}

\section{Double Ising model on a boundaryless surface of genus $g$}\label{sec:2}

In this section, we let $G$ be a graph embedded in a
compact, orientable, boundaryless surface
$\Sigma$ of genus $g$, and $G^*$ denote its dual graph. Since $\Sigma$
has no boundary, the first homology group
$\relhomol{1}$ of $\Sigma$ relative to its boundary is identified with the
first homology group $H_1(\Sigma,\ZZ/2\ZZ)$.

Instead of one Ising model on $\dual G$, we now consider two copies
of the Ising model, say a red one and a blue one, with the same coupling
constants $(J_{\dual e})$. These two models are not taken to be completely
independent: we require that they have the same defect conditions,
\emph{i.e.}, we ask that polygon configurations coming from the low-temperature
expansion of both spin configurations have the same homology class.

More precisely, from the point of view of the low-temperature expansion, 
we are interested in the probability measure $\PPdising$, on 
$
\Pdising:=\bigcup_{\eps\in\homol{1}}\P^{\eps}(G)\times\P^{\eps}(G),
$
defined by, for every $(P_{\text{red}},P_{\text{blue}})\in\Pdising$:
\begin{equation*}
\PPdising(P_{\text{red}},P_{\text{blue}}) =
  \frac{\C 
\Bigl(\prod_{e\in P_{\text{red}}}e^{-2J_{\dual e}} \Bigr)
  \Bigl( \prod_{e\in P_{\text{blue}}}e^{-2J_{\dual e}}
\Bigr)}{\Zdising(G^*,J)},
  \label{eq:doubleLT}
\end{equation*}
where $\C=\left( 2 \prod_{e\in E} e^{J_{\dual e}} \right)^2$, and
the partition function $\Zdising(G^*,J)$ is given by:
\begin{align*}
\Zdising(G^*,J)&=\sum_{\eps\in\homol{1}}\sum_{(P_{\text{red}},P_{\text{blue}}
)\in\P^ {
\eps}(G)\times\P^{\eps}(G)}\C\Bigl(\prod_{e\in P_{\text{red}}}e^{-2J_{\dual e}}
\Bigr)
  \Bigl( \prod_{e\in P_{\text{blue}}}e^{-2J_{\dual e}} \Bigr)\\
&=\sum_{\eps\in\homol{1}} (\Zising^\eps(J))^2.
\end{align*}

Given a pair $(P_{\text{red}},P_{\text{blue}})\in\Pdising$, and looking at the
superimposition $P_{\text{red}}\cup P_{\text{blue}}$ on $G$, one defines two new
edge configurations: 
\begin{itemize}
  \item $\mono(P_{\text{red}},P_{\text{blue}})$: consisting of monochromatic
edges of the superimposition $P_{\text{red}}\cup P_{\text{blue}}$, \emph{i.e.},
edges covered by exactly one of the polygon configuration;
  \item $\bi(P_{\text{red}},P_{\text{blue}})$: consisting of bichromatic edges
of the superimposition,
\emph{i.e.}, edges covered by both polygon configurations.
\end{itemize}
Edges which are not in the two configurations above are covered neither by
$P_{\text{blue}}$ nor by $P_{\text{red}}$. In Sections~\ref{sec:mono}
and~\ref{sec:bi} below, we characterize these two
sets of edges.

\subsection{Monochromatic edges}\label{sec:mono}

Let $\eps\in \homol{1}$, and consider a pair of polygon
configurations
$(P_{\text{red}},P_{\text{blue}})$ in $\P^{\eps}(G)\times\P^{\eps}(G)$. Then, it
can be realized as four pairs of Ising spin configurations
$(\pm\sigma,\pm\sigma')$, each with defect type $\eps$, where coupling constants
are negated along the representative $\rep{\epsilon}$ of $\eps$, chosen in Section
\ref{sec11}.

Following Wilson \cite{WilsonXOR}, to each of the four pairs of spin
configurations, one
assigns an XOR-\emph{spin configuration} defined as follows: at every vertex,
the XOR-spin is the product of the Ising-spins at that same vertex.

Note that the four pairs of spin configurations yield two distinct XOR-spin
configurations, one being obtained from the other by negating all spins. As a
consequence, both XOR-spin configurations have the same polygon configuration
separating clusters of $\pm 1$ spins, meaning that this polygon
configuration is independent of the choice of $(\pm\sigma,\pm\sigma')$
realizing $(P_{\text{red}},P_{\text{blue}})$, let us denote it by
$\XOR(P_{\text{red}},P_{\text{blue}})$. Note also, that
although the definition of $\sigma$
and $\sigma'$ depends on the particular
choice of representative $\rep{\epsilon}$, the XOR polygon configuration does
not: it is defined intrinsically from $(P_{\text{red}},P_{\text{blue}})$.

\begin{lem}\label{lem:mono}
For every pair of polygon configurations $(P_{\text{red}},P_{\text{blue}})\in\mathfrak{P}$,
the monochromatic edge configuration $\mono(P_{\text{red}},P_{\text{blue}})$
is exactly the XOR loop configuration $\XOR(P_{\text{red}},P_{\text{blue}})$. 
In particular, it is a polygon configuration of
$\P^{0}(G)$.
\end{lem}

\begin{proof}
Fix a pair of red and blue
polygon configurations$(P_{\text{red}},P_{\text{blue}})\in P^{\eps}(G)\times\P^{\eps}(G)$ for some
$\eps\in\homol{1}$. Let $(\sigma,
\sigma')$ be one of the four pairs of spin configurations whose low-temperature
expansion is $(P_{\text{red}},P_{\text{blue}})$. We need to show that, for every
edge $e$ of $G$, $e$ is monochromatic, if and only
if XOR-spins at vertices $u^*$, $v^*$ of the dual edge $e^*$ are distinct. 
Suppose that $e$ does not belong to $\rep{\epsilon}$. Then, 
\begin{align*}
&\text{- the edge } e \text{ is red only} \Leftrightarrow\ \sigma_{\dual u} \neq \sigma_{\dual v} \quad \text{and}\quad  \sigma'_{\dual u
} = \sigma'_{\dual v},\\
&\text{- the edge } e \text{ is blue only} \Leftrightarrow\ \sigma_{\dual u} = \sigma_{\dual v} \quad \text{and} \quad \sigma'_{\dual u }
\neq
  \sigma'_{\dual v}.
\end{align*}
If $e$ belongs to $\rep{\epsilon}$, the two above conditions hold with colors exchanged.
In all cases, $e$ is monochromatic if and only
if XOR spins at vertices $u^*$ and $v^*$ are distinct.

Being the boundary of some domain, the set of monochromatic edges must be a
polygon configuration of $G$, with homology class $0$.
\end{proof}

\subsection{Bichromatic edge configurations}\label{sec:bi}

Before describing features of bichromatic edge configurations, we recall some
general facts. A polygon configuration $P$
of the graph $G$ separates the surface~$\Sigma$ into $n_P$ connected components
$\Sigma_1,\ldots,\Sigma_{n_P}$, where $n_P\geq 1$. For every
$i\in\{1,\ldots,n_P\}$, $\Sigma_i$ is a surface of genus $g_i$ with boundary
$\partial\Sigma_i$. The boundary is either empty or consists of cycles
of~$\Sigma$.

As in Section~\ref{sec1}, $G_{\Sigma_i}$ denotes the subgraph of $G$, whose
vertex set $V_{\Sigma_i}$ is $V\cap\Sigma_i$, and whose edge set $E_{\Sigma_i}$
consists of edges of $E$ joining vertices of $V_{\Sigma_i}$, from which edges
on the boundary $\partial\Sigma_i$ are removed. The dual graph is denoted by
$G_{\Sigma_i}^*$.

Recall that $\relhomol[\Sigma_i]{1}$ denotes the first
homology group of $\Sigma_i$ relative to its boundary. Consider the morphism 
$\Pi_i=\Pi_{\Sigma,\Sigma_i}$, from $\homol[\Sigma]{1}$ to $\relhomol[\Sigma_i]{1}$ defined
as follows:
for every $\eps\in \homol[\Sigma]{1}$,
$\Pi_{i}(\eps)$ is the homology class in $\relhomol[\Sigma_i]{1}$
of the restriction of any representative $\rep{\epsilon}$ of $\eps$ to $\Sigma_i$, see
Appendix~\ref{app:A6} for details. 

The following lemma characterizes bichromatic edge configurations.

\begin{lem}\label{lem:bi}
Fix $\eps\in\homol[\Sigma]{1}$, and let $P\in\P^0(G)$ be a polygon
configuration, separating the surface $\Sigma$ into connected components
$\Sigma_1,\ldots,\Sigma_{n_P}$.
\begin{itemize}
\item If there exists a pair of polygon configurations
$(P_{\text{red}},P_{\text{blue}})\in\P^{\eps}(G)\times \P^{\eps}(G)$ such that
$\mono(P_{\text{red}},P_{\text{blue}})=P$; then, for every
$i\in\{1,\ldots,n_P\}$, the restriction of bichromatic edges to $G_{\Sigma_i}$
is the low-temperature expansion of
an Ising configuration on $G_{\Sigma_i}^*$, with coupling constants
$(2J_{e^*})$ and
defect condition $\Pi_{i}(\eps)$. As a consequence, it is a polygon
configuration in $\P^{\Pi_{i}(\eps)}(G_{\Sigma_i})$.
\item Given, for every $i\in\{1,\ldots,n_P\}$, a polygon configuration
$P_i\in\P^{\Pi_{i}(\eps)}(G_{\Sigma_i})$, there are $ 2^{n_P-1}$ pairs
$(P_{\text{red}},P_{\text{blue}})\in\P^{\eps}(G)\times \P^{\eps}(G)$ such that
$\mono(P_{\text{red}},P_{\text{blue}})=P$ and such that, for every
$i\in\{1,\ldots,n_P\}$, the restriction of bichromatic edges to $G_{\Sigma_i}$
is~$P_i$.
\end{itemize}
\end{lem}

\begin{proof}$\,$
\begin{itemize}
 \item Suppose that there exists a pair of polygon configurations
$(P_{\text{red}},P_{\text{blue}})$ of
$\P^{\eps}(G)\times \P^{\eps}(G)$ such
that $\mono(P_{\text{red}},P_{\text{blue}})=P$. Then, for every 
$i\in\{1,\ldots,n_P\}$, the restriction of bichromatic edges to $G_{\Sigma_i}$
exactly consists of the restriction to $G_{\Sigma_i}$ of one of two original
polygon configurations. Since this polygon configuration has homology
$\eps$ in
$\Sigma$, the homology class in $\relhomol[\Sigma_i]{1}$ of the restriction to
$\Sigma_i$ is
$\Pi_{i}(\epsilon)$ by definition. As a consequence, the bichromatic
edge configuration on $\Sigma_i$ is a polygon
configuration of $\P^{\Pi_{i}(\eps)}(G_{\Sigma_i})$.
Moreover, since all edges in the bichromatic configuration are present twice,
and since the weight of pairs of polygon configurations is the product of the
edge-weights contained in the pair of configurations, the effective weight of a
bichromatic edge $e$ is squared and becomes:
\begin{equation*}
  \left( e^{-2J_{\dual e}} \right)^2 = e^{-2 (2 J_{\dual e})},
\end{equation*}
which corresponds to a doubling of the coupling constants.
\item There are two spin
configurations, denoted by $\pm \xi$, whose low-temperature expansion is $P$.
Suppose that there exists
a pair
of spin configurations $(\sigma,\sigma')$ whose low-temperature expansion 
has $P$ as monochromatic edges,
then $\sigma\sigma'=\pm\xi$. Let us assume $\sigma\sigma'=\xi$, the
argument being similar in the other case, this has the effect of adding a global
factor 2 when speaking of spin configurations. The relation $\sigma\sigma'=\xi$
implies that there is freedom of choice
for exactly one spin configuration, say $\sigma$, the other being determined
by their product $\xi$. 

Consider a connected component
$\Sigma_i$, and a polygon configuration
$P_i\in\P^{\Pi_{i}(\eps)}(G_{\Sigma_i})$. We want $P_i$ to consist of
doubled edges, so that in particular, it must contain all red edges. There
are thus two
choices for the first spin configuration of $G_{\Sigma_i}^*$, denoted by
$\pm\sigma^i$. This holds for every $i\in\{1,\ldots,n_P\}$ and thus defines
$2^{n_P}$ spin configurations
$(\pm\sigma^1,\ldots,\pm\sigma^{n_P})$ of $G^*$. Recall that in
each of the $2^{n_P}$ cases, the second spin configuration is determined by the
condition $\sigma\sigma'=\xi$. Since on each connected
component~$\Sigma_i$, $\xi$ is identically equal to $\pm 1$, we deduce that
$(\sigma')^i=\pm\sigma^i$. As a consequence, the low-temperature expansion
of $\sigma'$
exactly consists of edges of $P_i$, \emph{i.e.}, $P_i$ consists of red
and blue edges. Summarizing, there are $2\cdot 2^{n_P}$ pairs of spin
configurations, or $2^{n_P-1}$ pairs of polygon configurations
$(P_{\text{red}},P_{\text{blue}})$, such that monochromatic edges are those of
$P$ and bichromatic edges those of $P_i$, $i\in\{1,\ldots,n_P\}$. Note that by
construction (choice of $\sigma^i$'s), each polygon configuration
$P_{\text{red}}$, $P_{\text{blue}}$ is in $\P^{\eps}(G)$.
\end{itemize}
\end{proof}

Consider a polygon configuration $P\in\P^{0}(G)$, and let $\eps\in
H_1(\Sigma,\ZZ/2\ZZ)$. Denote by $\Wdising^\eps[\mono=P]$ the contribution of
the
set 
$$
\{(P_{\mathrm{red}},P_{\mathrm{blue}})\in\P^{\eps}(G)\times\P^{\eps}(G)
:\mono(P_{\mathrm{red}},P_{\mathrm{blue}})=P\},
$$ 
to the partition function $(\Zising^\eps(J))^2$, and by 
$$
\Wdising[\mono=P]=\sum_{\eps\in\homol{1}}\Wdising^\eps[\mono=P].
$$

By the low-temperature expansion of the Ising partition function, the weight of
each polygon configuration
$P_{\text{red}}$, $P_{\text{blue}}$ is the product of edge-weights contained in
the configuration. As a consequence, the contribution of
$(P_{\text{red}},P_{\text{blue}})$ can be decomposed as a product over
monochromatic edges, and bichromatic edges of each of the components. Using
Lemmas~\ref{lem:mono} and~\ref{lem:bi}, this yields

\begin{prop}\label{cor:Zisingcarree}
For every polygon configuration $P\in\P^{0}(G)$ and every $\eps\in H_1(\Sigma,\ZZ/2\ZZ)$,
\begin{equation}\label{eq:Zisingcarree}
\Wdising^\eps[\mono=P]=2^{-1}\C\bigl(\prod_{e\in
P}e^{-2J_{e^*}}\bigr)\Bigl(\prod_{i=1}^{n_P}
 2 \ZLT^{\Pi_{i}(\epsilon)}(G_{\Sigma_i},2 J)\Bigr),
\end{equation}
where $\C=\bigl(2\prod_{e\in E}e^{J_{e^*}}\bigr)^2$.
Moreover, the double Ising partition function can be rewritten as:
$$
\Zdising(J)=\sum_{P\in\P^{0}(G)}\Wdising[\mono=P],$$
and the probability measure
$\PPdising$ induces a probability measure on 
polygon configurations of $\P^{0}(G)$, given by:
\begin{equation}\label{equ:Pisingmono}
\forall\,P\in\P^{0}(G),\quad\PPdising[\mono=P]=\frac{\Wdising[\mono=P]}{
\Zdising(J)
}.
\end{equation}
\end{prop}

\section{Mixed contour expansion}\label{sec:mixedcontour}

In \cite{Nienhuis}, Nienhuis rewrites the partition function of the
Ashkin--Teller model on the square lattice as a statistical sum over polygon
families on $G$ and $G^*$ which do not intersect. We apply the same approach
to the double Ising model on $\Sigma$ but some care is required to keep track
of the homology class of the polygon configurations involved.

We fix $\eps\in\homol{1}$
and a polygon configuration $P\in\P^{0}(G)$. In Proposition
\ref{thm:mixedcontour1}, we apply the
low/high-temperature duality to each of the terms
$\ZLT^{\Pi_{i}(\epsilon)}(G_{\Sigma_i},2 J)$ involved in the
expression of $\Wdising^{\eps}[\mono=P]$ of Equation \eqref{eq:Zisingcarree}.
This has the effect of transforming bichromatic polygon configurations of
$G_{\Sigma_i}$ into dual polygon configurations of $G_{\Sigma_i}^*$. 
Then, in Proposition~\ref{prop:mixedcontour}, we sum
over $\eps\in\homol{1}$, and show that the outcome simplifies
to a sum over dual polygon configurations of the dual graph, having 0 homology class
in $\homol{1}$, and not intersecting $P$.

\begin{prop}\label{thm:mixedcontour1}
For every polygon configuration $P\in\P^{0}(G)$ and every $\eps\in H_1(\Sigma,\ZZ/2\ZZ)$,
\begin{multline*}
\Wdising^{\eps}[\mono=P]=
\C'
\Bigl(\prod_{e\in P}
\frac{2e^{-2J_{e^*}}}{1+e^{-4J_{e^*}}}\Bigr)\times\\
\times \prod_{i=1}^{n_P}\Bigl[\sum_{\tau^i\in
H_1(\Sigma_i,\ZZ/2\ZZ)}(-1)^{(\tau^i|\Pi_{i}(\eps))}
\sum_{P_i^*\in\P^{\tau^i}(G_{\Sigma_i}^*)}\Bigl(\prod_{e^*\in P_i^*}
\frac{1-e^{-4J_{e^*}}}{1+e^{-4J_{e^*}}}
\Bigr)\Bigr],
\end{multline*}
where, $\C'=2^{|V^*|+1}\bigl(\prod_{e\in
E}\cosh (2J_{\dual e})\bigr)$.
\end{prop}

\begin{proof}
The expression for $\Wdising^{\eps}[\mono=P]$ of Equation
\eqref{eq:Zisingcarree}
can be rewritten as:
\begin{equation*}
\Wdising^\eps[\mono=P]=2^{n_P-1}\C\bigl(\prod_{e\in
P}e^{-2J_{e^*}}\bigr)\Bigl(\prod_{i=1}^{n_P}
\ZLT^{\Pi_{i}(\epsilon)}(G_{\Sigma_i},2 J)\Bigr).
\end{equation*}

For every $i\in\{1,\ldots,n_P\}$, the contribution,
$\ZLT^{\Pi_{i}(\eps)}(G_{\Sigma_i}, 2J)$
is the low-temperature expansion of an Ising model on vertices of
$V_{\Sigma_i}^*$ with coupling constants $2J_{e^*}$ and defect condition
$\Pi_{i}(\eps)$. Using the relation between Kramers
and Wannier's low and high-temperature expansions of
\eqref{eq:lowhighexp}, it can be expressed as:
\begin{equation*}
  \ZLT^{\Pi_{i}(\eps)}(G_{\Sigma_i}, 2J) = \A_i \times
  \sum_{\tau^i\in\homol[\Sigma_i]{1}}(-1)^{(\tau^i| \Pi_{i}(\eps))}
  \ZHT^{\tau^i}(G_{\Sigma_i}, 2J),
\end{equation*}
where
\begin{equation*}
  \A_i= 2^{|V_{\Sigma_i}^*|-1} \prod_{e\in E_{\Sigma_i}} \frac{\cosh
    (2J_{e^*})}{e^{2J_{e^*}}}.
\end{equation*}
Let us first compute the part which is independent of $\eps$. Observing that the
collection of sets of dual vertices  $(V_{\Sigma_i}^*)_{i=1}^{n_P}$ is a partition
of $V^*$, one writes:
\begin{align*}
2^{n_P-1}  \C  \bigl( \prod_{e\in P}e^{-2J_{e^*}} \bigr) \bigl(\prod_{i=1}^{n_P}
\A_i\bigr)
&=2^{n_P-1} 2^2\bigl(\prod_{e\in E}e^{2J_{e^*}}\bigr)
\bigl( \prod_{e\in P}e^{-2J_{e^*}} \bigr) 
2^{|V^*|-n_P}
\bigl(\prod_{e\in
E_{\Sigma_i}}\frac{\cosh(2J_{e^*})}{e^{2J_{e^*}}}\bigr).
\end{align*}
Noticing that
the collection of edges in the $\Sigma_i$'s is exactly the set of edges of $G$ not in
$P$, we have:
\begin{equation*}
2^{n_P-1}  \C  \bigl( \prod_{e\in P}e^{-2J_{e^*}} \bigr) \bigl(\prod_{i=1}^{n_P}
\A_i\bigr)
  =
  2^{|V^*|+1} \bigl(\prod_{e\in E} \cosh (2J_{e^*})\bigr)
  \bigl(\prod_{e\in P} \cosh (2 J_{e^*})^{-1} \bigr).
\end{equation*}
Define the constant $\C'=2^{|V^*|+1}\bigl(\prod_{e\in E}
\cosh(2J_{e^*})\bigr)$, then by definition of $\ZHT^{\tau^i}(G_{\Sigma_i},2J)$,
one deduces that $\Wdising^{\eps}[\mono=P]$ equals:
\begin{align*}
\C'
\bigl(\prod_{e\in P}
\cosh(2J_{e^*})^{-1}\bigr)
\prod_{i=1}^{n_P}\Bigl[
\sum_{\tau^i\in
H_1(\Sigma_i,\ZZ/2\ZZ)}(-1)^{(\tau^i|\Pi_{i}(\eps))}
\sum_{P_i^*\in\P^{\tau^i}(G_{\Sigma_i}^*)}\bigl(\prod_{e^*\in P_i^*}\tanh
(2J_{e^*})\bigr)
\Bigr].
\end{align*}
The proof of Proposition \ref{thm:mixedcontour1} is concluded by observing that
$$
\cosh(2J_{\dual e})^{-1}=\frac{2e^{-2J_{e^*}}}{1+e^{-4J_{\dual e}}},\text{ and
}\;\;
\tanh(2J_{e^*})=\frac{1-e^{-4J_{e^*}}}{1+e^{-4J_{e^*}}}.\qedhere
$$
\end{proof}

In order to have an explicit expression for the contribution 
of $P$ to the partition function $\Zdising(J)$, we need to sum the quantities
$\Wdising^\eps[\mono=P]$ of Proposition~\ref{thm:mixedcontour1} over
$\eps\in\homol{1}$. This is the
object of the next proposition.

\begin{prop}\label{prop:mixedcontour2}
For every polygon configuration $P\in\P^0(G)$,
  \begin{equation*}
    \Wdising[\mono=P]= \C_{\mathrm{I}}\,\Bigl( \prod_{e\in
P}\frac{2e^{-2J_{\dual
    e}}}{1+e^{-4J_{\dual e}}} \Bigr) 
\sum_{\{P^*\in\P^0(G^*):\,P\cap P^*=\emptyset\}}
\Bigl(    \prod_{\dual e  \in \dual P} 
\frac{1-e^{-4J_{\dual e}}}{1+e^{-4J_{\dual e}}} \Bigr).
  \end{equation*}
  where $\C_{\mathrm{I}}=2^{2g}\mathcal{C}'=2^{|V^*|+2g+1}\left(\prod_{e\in E}\cosh
(2J_{e^*})\right)$.
\end{prop}

\begin{proof}
To simplify notation, let us write the product of weights of edges in
polygon configurations as follows:
\begin{equation*}
\Theta(P)=\prod_{e\in P}
\frac{2e^{-2J_{e^*}}}{1+e^{-4J_{e^*}}}
,\quad\dual\Theta(P_i^*)=\prod_{e^*\in P_i^*}
\frac{1-e^{-4J_{e^*}}}{1+e^{-4J_{e^*}}}
,\; \text{ for }i\in\{1,\ldots,n_P\}.
\end{equation*}
Then,
\begin{align*}
\Wdising[\mono=P]=\C'\Theta(P)\!\!
\sum_{\eps\in \homol{1}}
\prod_{i=1}^{n_P}
\Bigl[\sum_{\tau^i\in H_1(\Sigma_i,\ZZ/2\ZZ)}
(-1)^{(\tau^i|\Pi_{i}(\eps))}\!\!
\sum_{P_i^*\in\P^{\tau^i}(G_{\Sigma_i}^*)}\dual\Theta(P_i^*)\Bigr].
\end{align*}
Expanding the product over $i\in\{1,\ldots,n_P\}$ and 
exchanging the summation over $\eps$ and $(\tau^1,\ldots,\tau^{n_P})$,
one obtains that $\Wdising[\mono=P]$ is equal to:
\begin{align*}
&\C'\Theta(P)\!\!
\sum_{(\tau^1,\ldots,\tau^{n_P})\in 
\prod_{i=1}^{n_P}H_1(\Sigma_i,\ZZ/2\ZZ)}
\Bigl(\sum_{\eps\in \homol{1}}(-1)^{\sum_{i=1}^{n_P}
(\tau^i|\Pi_{i}(\eps))}\Bigr)\prod_{i=1}^{n_P}
\Bigl(\sum_{P_i^*\in\P^{\tau^i}(G_{\Sigma_i}^*)}\dual\Theta(P_i^*)\Bigr).
\end{align*}
The evaluation of the intersection form
 $(\tau^i | \Pi_{i}(\eps))$ on
$\homol[\Sigma_i]{1}\times\relhomol[\Sigma_i]{1}$ is equal to $(\pi_i(\tau^i) |
\eps)$ on $\homol{1}\times\homol{1}$, where $\pi_i=\pi_{\Sigma,\Sigma_i}$ is the
projection induced by the inclusion $\Sigma_i\subset \Sigma$, see Appendix
\ref{app:A6}. Indeed, take a
representative $\rep{\tau}^i$ of $\tau^i$ in $\Sigma_i$. Counting
intersections with the
restriction of a representative $\rep{\epsilon}$ of $\eps$ to $\Sigma_i$ is the same as
counting
intersections with the whole $\rep{\epsilon}$, since $\rep{\tau}^i$ is confined to
$\Sigma_i$ and thus has no intersection with $\rep{\epsilon}$ outside of
$\Sigma_i$. Therefore,
\begin{equation*}
  \sum_{i=1}^{n_P} (\tau^i|\Pi_{i}(\eps))=
  \Bigl(\sum_{i=1}^{n_P} \pi_i(\tau^i)\Big|\eps\Bigr).
\end{equation*}
An application of the orthogonality relation~\eqref{eq:orthogonality} in the
special case where $\partial\Sigma=\emptyset$ gives that the sum over $\epsilon$
is $0$ unless $\sum_{i=1}^{n_P}
\pi_i(\tau^i)=0\in\homol{1}$, \emph{i.e.}, the homology class on the surface
$\Sigma$ of the whole dual polygon configuration $\dual P = P_1^* \cup \cdots \cup
P_{n_P}^*$ is zero, \emph{i.e.}, $\dual P \in \P^{0}(\dual G)$, and that in this
case, the sum over $\epsilon$ equals $2^{2g}$.
\end{proof}

As an interesting corollary, we obtain a mixed contour expansion for the
partition function of the double Ising model, and the corresponding expression
for the double Ising probability measure of monochromatic polygon
configurations which, by Lemma \ref{lem:mono}, are the XOR polygon
configurations.

\begin{cor}\label{prop:mixedcontour}$\,$
\begin{itemize}
 \item 
The double Ising partition function can be rewritten as:
  \begin{equation*}
    \Zdising(G^*,J) =\C_{\mathrm{I}}\sum_{\substack{\{(P,P^*)\in \P^{0}(G)\times
\P^{0}(\dual G):\, 
  P\cap \dual P=\emptyset\} }} \Bigl(\prod_{e\in P} 
 \frac{2e^{-2J_{\dual
    e}}}{1+e^{-4J_{\dual e}}} \Bigr) 
    \Bigl(\prod_{\dual e  \in \dual P}  \frac{1-e^{-4J_{\dual
    e}}}{1+e^{-4J_{\dual e}}} \Bigr),
  \end{equation*}
where $\C_{\mathrm{I}}=
2^{|V^*|+2g+1} \left(\prod_{e\in E}\cosh(2J_{\dual e})\right)$. 
\item 
For every
dual polygon configuration $P\in\P^{0}(G)$:
$$
\PPdising[\XOR=P]=\frac{\displaystyle\Bigl( \prod_{e\in
P}\frac{2e^{-2J_{\dual
    e}}}{1+e^{-4J_{\dual e}}} \Bigr) 
\sum_{\{P^*\in\P^0(G^*):\,P\cap P^*=\emptyset\}}
\Bigl(    \prod_{\dual e  \in \dual P} 
\frac{1-e^{-4J_{\dual e}}}{1+e^{-4J_{\dual e}}}\Bigr)}
{\displaystyle\!\!\!
  \sum_{\{(P,P^*)\in \P^{0}(G)\times\P^{0}(\dual G):\,
  P\cap \dual P=\emptyset\} } 
 \Bigl(\prod_{e\in P} \frac{2e^{-2J_{\dual e}}}{1+e^{-4J_{\dual e}}} \Bigr)
 \Bigl(\prod_{\dual e \in \dual P}\frac{1-e^{-4J_{\dual e}}}{1+e^{-4J_{\dual e}}} \Bigr)}.
$$
\end{itemize}
\end{cor}

Note that one can see  Kramers and Wannier's duality on this expression: the
duality relation between coupling constant
\begin{equation*}
  \tanh \dual J = e^{-2 J}
\end{equation*}
exchanges the expression for an edge of $P$ and a dual edge of $\dual P$.

\section{Quadri-tilings and polygon configurations}\label{sec5}

In this section, we let $G$ be a graph embedded in a
compact, orientable, boundaryless surface $\Sigma$ of genus $g$, and $G^*$ denote
its dual graph. For the moment, we forget about the double Ising model. The
goal of this section is to explicitly construct pairs of
non-intersecting polygon configurations of $G$ and $G^*$, from a
dimer model on a decorated, bipartite version $\GQ$ of $G$, called  
\emph{quadri-tilings}~\cite{Bea1}.

This construction is done in two steps. The first step uses a mapping of
Nienhuis \cite{Nienhuis}, which constructs pairs of non-intersecting primal and
dual polygon configurations, from 6-vertex configurations of the medial graph;
this is the subject of Section \ref{sec31}. The second step consists of using
Wu--Lin/Dub\'edat's mapping \cite{WuLin,Dubedat} from the 6-vertex model of the
medial graph to the bipartite dimer model on the decorated graph $\GQ$. This is
the subject of
Section \ref{sec:52}.

Using the above results and those of Section \ref{sec:mixedcontour},
Theorem \ref{thm:main1inside} proves that $\XOR$ loops of the double Ising
model have the same law as primal polygon configurations of the bipartite dimer
model.

\subsection{6-vertex model and polygon configurations}\label{sec31}

The \emph{medial graph}
$\GM$ of the graph $G$ is defined as follows. Vertices of $\GM$
correspond to edges of $G$. Two vertices of the medial graph are joined by
an edge if the corresponding edges in the primal graph are incident. Observe
that $\GM$ is also the medial graph of the dual graph $G^*$,
and that vertices of the medial graph all have degree four. Figure
\ref{fig:medial} represents the medial graph of a subset of $\ZZ^2$.

\begin{figure}[ht]
\begin{center}
\includegraphics[height=4cm]{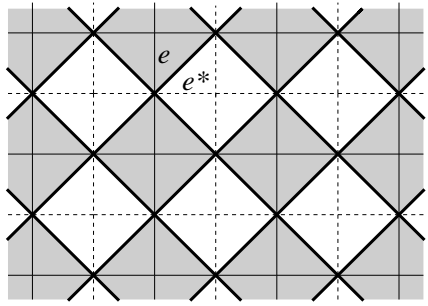}
\caption{The medial graph of $\ZZ^2$: plain lines represent $\ZZ^2$, dotted
lines represent the dual graph $\ZZ^2$, and thick plain lines represent the
medial graph $(\ZZ^2)^M$. Grey (resp. white) faces of the medial graph correspond to primal (resp. dual) vertices of the initial graph.}
\label{fig:medial}
\end{center} 
\end{figure}

A \emph{6V-configuration} or an \emph{ice-type configuration} is an
orientation of edges of $\GM$, such that every vertex has exactly two incoming
edges \cite{Lieb}. An equivalent way of defining $6$-vertex configurations
uses edge configurations instead of orientations, as represented in Figure
\ref{fig:6vconfigs}. This approach is more
useful in our context, so that we define a \emph{$6$-vertex configuration}
to be an edge configuration, such that around every vertex of
$\GM$, there is an even number of consecutive present edges.

\begin{figure}[ht]
\begin{center}
\includegraphics[width=\textwidth]{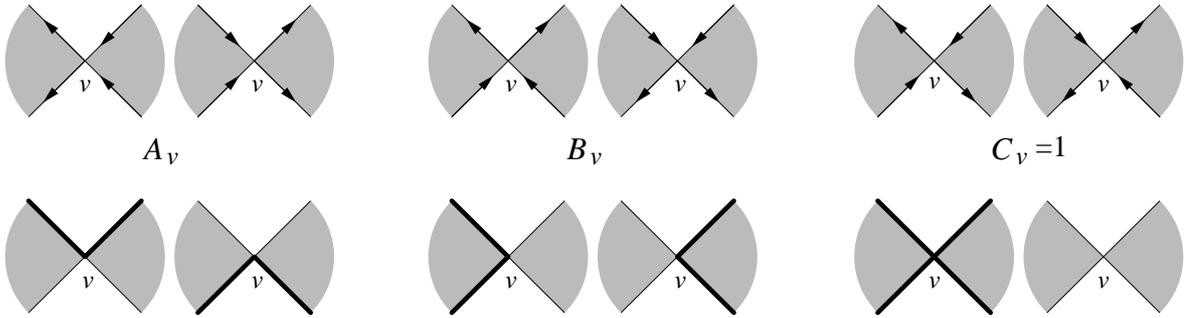}
\caption{The six possible local configurations around a vertex $v$ of the
6-vertex model, and their respective weights: arrow representation (top),
and the even degree subgraph representation (bottom).}
\label{fig:6vconfigs}
\end{center} 
\end{figure}
In order to make the 6-vertex model a model of statistical mechanics, weights
are associated to local configurations around a vertex, and the probability of a
6-vertex configuration is taken to be proportional to the product of its local
weights. In absence of external field, the weights of complementary local
configurations are taken to be equal: there are thus three
parameters for each vertex $v$ of the medial graph $\GM$, denoted by $A_v$,
$B_v$ and $C_v$. Since multiplying
these three parameters by the same positive constant does not change the
measure, we set $C_v=1$, see also Figure \ref{fig:6vconfigs}. Let us denote by
$\Z6V(\GM,(A,B))$ the partition function of this model.

\textbf{Mapping I} \cite{Nienhuis}. Consider the following combinatorial mapping
from 6-vertex
configurations to
edge configurations of the primal and dual graphs: whenever a vertex of $\GM$ has
two
neighboring edges in the 6 vertex configuration, put the edge of $G$ or $\dual
G$ separating the present and the absent edges; see Figure
\ref{fig:6vertexMapping}. The following lemma characterizes this mapping, see
also \cite{Nienhuis}.

\begin{figure}[ht]
\begin{center}
\includegraphics[width=10cm]{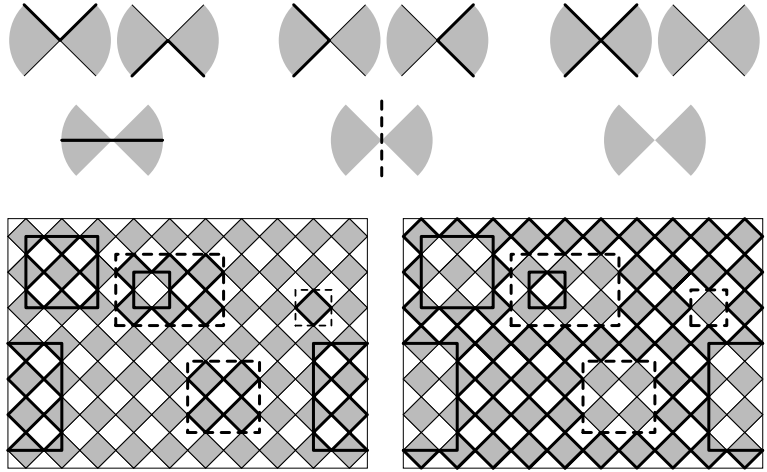}
\caption{Mapping on the local level (top), and on the global level (bottom).}
\label{fig:6vertexMapping}
\end{center} 
\end{figure}

\begin{lem}\label{lem:mapping6V}$\,$
\begin{itemize}
\item Mapping I associates to a 6-vertex configuration a pair of polygon
configurations $(P,P^*)$, which do not intersect and such that the homology
class of $P\cup P^*$ in $H_1(\Sigma;\ZZ/2\ZZ)$ is $0$.
\item Given a pair of polygon configurations $(P,P^*)$ as above, there are
exactly two 6-vertex configurations which are mapped to $(P,P^*)$.
\end{itemize}
\end{lem}
\begin{proof}
A 6-vertex configuration consists of clusters of present/absent edges, see
Figure \ref{fig:6vertexMapping}. The primal/dual edge configuration assigned by
the mapping
consists of the boundary of those clusters. It is thus a polygon
configuration with $0$ homology class. A primal and the corresponding dual
edge
cannot intersect, since that would correspond to a configuration on $\GM$ where
around a vertex there is alternatively one edge present, one edge absent, then
again one present, one absent, which is a forbidden local configuration for the
6-vertex model.

Conversely suppose we are given a pair of polygon configurations $(P,P^*)$ as
above. The fact
that the homology class of $P\cup P^*$ in $H_1(\Sigma;\ZZ/2\ZZ)$ is $0$, 
exactly means that it is the boundary of
domains that can be painted alternatively in two colors consistently. Put all
edges of $\GM$ in domains of one color, and remove all edges in domains
of the other color. In this way, one exactly obtains two valid
configurations of the
6-vertex model, one being the complement of the other, depending on which color
is used to represent present edges, see also Figure \ref{fig:6vertexMapping}.
\end{proof}

Let us denote by $\P^{0}(G\cup G^*)$ the set of pairs $(P,P^*)$
of polygon configurations of $G$ and $G^*$ respectively, such
that the union $P\cup P^*$ has $0$ homology class in $\homol{1}$.

Recall that to every edge $e$ of the primal graph (resp. $e^*$ of the dual
graph), corresponds a vertex $v$ of the medial graph, which we denote by $v(e)$
(resp. $v(e^*)$). Using this fact, one can naturally define a weight function
on edges of $G$ and $G^*$:
\begin{equation*}
\forall e\in G,\quad a_e:=A_{v(e)};\quad  \forall e^*\in G^*,\quad
b_{e^*}:=B_{v(e^*)}.
\end{equation*}
With this choice of weights and using Lemma \ref{lem:mapping6V}, we obtain the
following.
\begin{lem}\label{lem:6V}
\begin{equation*}
\Z6V(\GM,(A,B))=2\sum_{\{(P,P^*)\in\P^{0}(G\cup
G^*):\,P\cap
P^*=\emptyset\}}\prod_{e\in P}a_{e}\,\prod_{e^*\in P^*}b_{e^*}. 
\end{equation*}
\end{lem}

\subsection{Quadri-tilings and 6-vertex model}\label{sec:52}

Let us define yet another graph built from the graph $G$. The
\emph{quadri-tiling} graph of $G$, denoted by $\GQ$, is the 
decorated graph obtained from $\GM$ by
replacing every vertex by a decoration which is a quadrangle. The graph $\GQ$ is
bipartite and can be drawn on the same surface as $G$. Edges shared by $\GQ$
and $\GM$ are referred to as \emph{external} edges, and those inside the
decorations as \emph{internal}.

\begin{figure}[ht]
\centering
  \includegraphics[height=2cm]{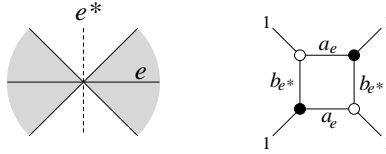}
  \caption{Decoration of a vertex of the medial graph $\GM$ (left) to obtain a
piece  of the quadri-tiling graph $\GQ$ (right).}
  \label{fig:medial-quadri}
\end{figure}

A \emph{dimer configuration} of $\GQ$, also known as a
\emph{perfect matching} is a spanning subgraph of $\GQ$ where every vertex has
degree
exactly one. Let us
denote by $\M(\GQ)$ the set of dimer configurations of the graph $\GQ$.
In a particular decoration of $\GQ$, a dimer configuration of $G$ looks like one
of the seven possibilities represented in Figure~\ref{fig:local_quadri} (top).

Assigning positive weights $(w_e)_{e\in \EQ}$ to edges of $\GQ$,
the \emph{dimer Boltzmann measure}, denoted $\PPquadri$, is defined by:
$$
\forall M\in\M(\GQ),\quad\PPquadri(M)=\frac{\prod_{e\in M}w_e}{\Zquadri(\GQ,w)},
$$
where $\Zquadri(\GQ,w)=\sum_{M\in \M(\GQ)}\prod_{e\in M}w_e$ is the dimer
partition function. This defines a model of
statistical mechanics, called the \emph{dimer model} on $\GQ$.

\textbf{Mapping II} \cite{WuLin,Dubedat}. Requiring exterior edges to match
yields a
mapping from
dimer configurations of $\GQ$ to 6-vertex configurations of $\GM$; see Figure
\ref{fig:local_quadri}. This mapping between local configurations is one-to-one except in the empty edge case
where this mapping is two-to-one.

\begin{figure}[ht]
  \centering
  \includegraphics[width=12cm]{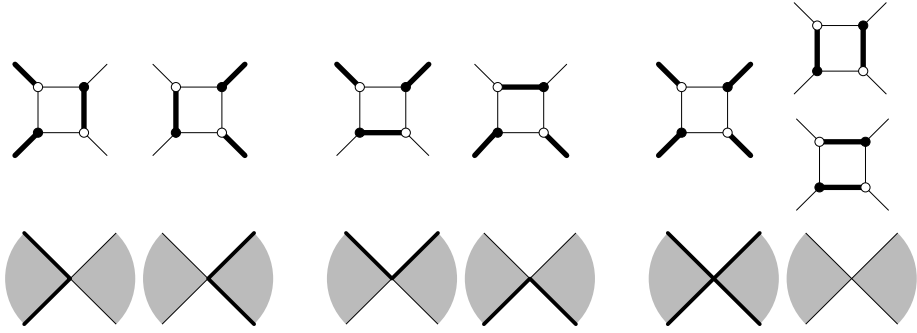}
  \caption{The local configurations of a quadri-tiling and the corresponding
  local 6-vertex configurations.}
  \label{fig:local_quadri}
\end{figure}

We now choose weights of edges in a specific way. Let $A$ and
$B$ be positive functions on vertices of the medial graph $\GM$, defining
weights of local 6-vertex configuration; and let $a$ and $b$ be the induced
weight functions on edges of $G$ and $G^*$ respectively. The
weight function $w^{a,b}$ on edges of $\GQ$ is defined as follows, see also
Figure \ref{fig:medial-quadri}:
\begin{equation*}
w^{a,b}_e=
\begin{cases}
1&\text{if $e$ is an external edge}\\
a_{\bar{e}}&\text{if $e$ is an interior edge, parallel to a primal edge
$\bar{e}$} \\
b_{\bar{e}^*}&\text{if $e$ is an interior edge, parallel to a dual edge
$\bar{e}^*$}.
\end{cases}
\end{equation*}
Let us denote by $\Zquadri(\GQ,(A,B))$ the corresponding partition function.
From now
on, we suppose that local weights of the 6-vertex model satisfy
the relation:
\begin{equation*}
\forall v\in V(\GM),\quad A_v^2 +B_v^2 =1.
\end{equation*}
This implies that, $\forall v\in V(\GM),\;  \Delta_v = \frac{A_v^2 + B_v^2 -
C_v^2 }{2 A_v B_v} = 0,$ \emph{i.e.}, the model is \emph{free-fermionic}.

Using the weight functions $a$ and $b$, the free-fermionic condition can be
rewritten as:
\begin{equation}\label{equ:freefermion}
\forall e\in E,\quad a_e^2+b_{e^*}^2=1.
\end{equation}
With this choice of weight, we thus obtain the following.
\begin{lem}\label{lem:quadri}
When local weights of the 6-vertex model satisfy the free-fermionic relation,
Mapping~II is
weight preserving and,
$$
\Zquadri(\GQ,(A,B))=\Z6V(\GM,(A,B)).
$$ 
\end{lem}

Let us now choose weights of edges of $\GQ$ to depend on coupling constants
$(J_{e^*})$ of the double Ising model:
\begin{equation}\label{equ:quadweights}
\forall e\in E,\;a_e=
\frac{2e^{-2J_{e^*}}}{1+e^{-4J_{e^*}}},\,\text{ and }\;\;
\forall e^*\in E^*,\;b_{e^*}=\frac{{1-e^{-4J_{e^*}}}}{1+e^{-4J_{e^*}}}.
\end{equation}
Then,
it is straightforward to check that $a$ and $b$ satisfy the free-fermionic condition~\eqref{equ:freefermion}:
\begin{align*}
\forall e\in E,\;
a_e^2+b_{e^*}^2=\frac{4e^{-4J_{e^*}}+(1-e^{-4J_{e^*}})^2}{(1+e^{-4J_{e^*}})^2}
=1. 
\end{align*}

It should be noted that in the more general case of the Ashkin--Teller model,
when spins of the two Ising models interact, the mapping with the 6-vertex model
still holds~\cite{Nienhuis,Saleur}, but the model is not free-fermionic anymore.

Since weights $a$ and $b$ depend on coupling constants $(J_{e^*})$, we set
$J$ as argument for the corresponding dimer model partition function.

Recall that the mixed contour expansion of the double Ising partition  function,
see Corollary~\ref{prop:mixedcontour}, involves pairs of non-intersecting primal
and dual polygon configurations, each of which has
0 homology class in $\homol{1}$. We want to take the same restriction here.

Consider a dimer configuration $M$ of the graph $\GQ$, then Mapping II assigns
to $M$
a 6-vertex configuration, and Mapping I assigns to this 6-vertex configuration
a pair of non-intersecting polygon configuration of $\P^{0}(G\cup G^*)$. Let us
denote this pair by $\poly(M)=(\poly_1(M),\poly_2(M))$. 

We restrict ourselves to
dimer configurations $M$ such that $\poly_1(M)$ has 0 homology class, and denote
by
$\M^0(\GQ)$ this set. Note that since
the superimposition $\poly_1(M)\cup\poly_2(M)$ has $0$ homology class, this
automatically implies that $\poly_2(M)$ also has $0$ homology. Let $\PPquadri^0$
be the corresponding dimer Boltzmann-measure, and $\Zquadri^0(\GQ,J)$
be the corresponding partition function.

Let $(P,P^*)\in \P^0(G)\times\P^0(G^*)$ such that $P\cap P^*=\emptyset$. Denote
by $\Wquadri[\poly=(P,P^*)]$ the contribution of the set: 
$$
\{M\in\M(\GQ):\,\poly(M)=(P,P^*)\}\subset \M^0(G),
$$ 
to the partition function $\Zquadri^0(\GQ,J)$. Then, as a consequence of
Lemmas \ref{lem:6V} and \ref{lem:quadri}, we have:

\begin{prop}\label{prop:quadritilings}
When weights assigned to edges of the graph $\GQ$ are chosen as in Equation
\eqref{equ:quadweights}, we have for all $(P,P^*)\in \P^0(G)\times \P^0(G^*)$,
such that $P\cap P^*=\emptyset$:
$$
\Wquadri[\poly=(P,P^*)]=2\Bigl(\prod_{e\in P} 
 \frac{2e^{-2J_{\dual
    e}}}{1+e^{-4J_{\dual e}}} \Bigr) 
    \Bigl(\prod_{\dual e  \in \dual P}  \frac{1-e^{-4J_{\dual
    e}}}{1+e^{-4J_{\dual e}}} \Bigr),
$$
Moreover, the dimer model partition function can be written as:
$$
\Zquadri^0(\GQ,J)=2\sum_{\substack{\{(P,P^*)\in
\P^{0}(G)\times
\P^{0}(\dual G):\, 
  P\cap \dual P=\emptyset\} }} \Bigl(\prod_{e\in P} 
 \frac{2e^{-2J_{\dual
    e}}}{1+e^{-4J_{\dual e}}} \Bigr) 
    \Bigl(\prod_{\dual e  \in \dual P}  \frac{1-e^{-4J_{\dual
    e}}}{1+e^{-4J_{\dual e}}} \Bigr)
$$
and the probability measure $\PPquadri^0$ induces a probability measure on
polygon configurations of $\P^0(G)$, given by:
$$
\PPquadri^0[\poly_1=P]=\frac{\Bigl( \prod_{e\in
P}\frac{2e^{-2J_{\dual
    e}}}{1+e^{-4J_{\dual e}}} \Bigr) 
\sum\limits_{\{P^*\in\P^0(G^*):\,P\cap P^*=\emptyset\}}
\Bigl(    \prod_{\dual e  \in \dual P} 
\frac{1-e^{-4J_{\dual e}}}{1+e^{-4J_{\dual e}}}
\Bigr)}{\sum\limits_{\substack{\{(P,P^*)\in \P^{0}(G)\times
\P^{0}(\dual G):\, 
  P\cap \dual P=\emptyset\} }} \Bigl(\prod_{e\in P} 
 \frac{2e^{-2J_{\dual
    e}}}{1+e^{-4J_{\dual e}}} \Bigr) 
    \Bigl(\prod_{\dual e  \in \dual P}  \frac{1-e^{-4J_{\dual
    e}}}{1+e^{-4J_{\dual e}}} \Bigr)}.
$$
\end{prop}

Combining Corollary \ref{prop:mixedcontour} and Proposition
\ref{prop:quadritilings} yields the following.

\begin{thm}\label{thm:main1inside}
$\,$
\begin{itemize}
\item The double Ising partition function and the dimer model partition function
are equal up to an explicit constant:
$$
\Zdising(G^*,J)=2^{|V^*|+2g}\Bigl(\prod_{e\in
E}\cosh(2J_{e^*})\Bigr)\Zquadri^0(\GQ,J).
$$
\item XOR-polygon configurations of the double Ising model on
$G^*$ have the same law as $\poly_1$ configurations of the corresponding dimer
model on the bipartite graph $\GQ$:
$$
\forall P\in\P^{0}(G),\;\PPdising[\XOR=P]=\PPquadri^0[\poly_1=P].
$$
\end{itemize}
\end{thm}
Note that the first part of Theorem \ref{thm:main1inside} can also be deduced
from the results of \cite{Dubedat} and \cite{CimasoniDuminil}.

Since the quadri-tiling model is a bipartite dimer model, it can be studied in
great detail using the tools of Kasteleyn theory, of which we recall some elements of in
the next subsection. These tools can be used to study the distribution of the
XOR Ising configurations.

\subsection{Kasteleyn theory}
\label{sec:quadri_kast}

We now recall some elements of the Kasteleyn theory for bipartite dimer models
and apply it to the dimer model on $\GQ$. This simplified version for
bipartite graphs of the more general theory developed by Temperley and Fisher, Kasteleyn
\cite{TemperleyFisher,Kasteleyn} is due to Percus \cite{Percus}. The main tool is the
\emph{Kasteleyn matrix}, defined as follows in the bipartite case:
\begin{itemize}
  \item rows (resp. columns) are indexed by white (resp. black) vertices;
  \item the absolute value of an entry is 0 if the corresponding white and black
    vertices are not adjacent, and is the dimer weight of the edge formed
by these vertices when they are adjacent;
  \item signs of the entries are chosen in such a way that around all (bounded)
faces, the number of minus signs around a face has the same parity as half of
the degree
    of the face, minus~1.
\end{itemize}
If the graph $\GQ$ is planar, the partition
function of the model is, up to a global sign, the determinant of the Kasteleyn
matrix. Indeed, when expanding the
determinant of $K$ as a sum over permutations, the only non-zero
terms are those corresponding to dimer configurations, and their absolute
value is the correct weight (see for example \cite{Kenyon6} p.3). The third condition about signs is here to
compensate the signatures of the permutations, so that all the terms exactly 
have the same sign. Kasteleyn showed \cite{Kast61,Kasteleyn} that such a choice
of signs exists and is essentially unique: changing the sign of each edge around a
particular vertex still yields a choice of signs satisfying the third condition,
and one can pass from one valid choice to another by a succession of such
operations. In order to get the right global sign, one can choose a reference
dimer configuration $M_0$, giving a bijection between white and black vertices, agree
that the order chosen for rows and columns of $K$ is compatible with this
bijection, and choose signs so that all entries of $K$ corresponding to dimers of
the reference
configuration have sign $+$. 

If the graph $\GQ$ is embedded in a surface $\Sigma$ of genus $g>0$, the Kasteleyn theory is 
more complicated. There is a topological obstruction for the existence of a sign distribution on
edges giving every dimer configuration a $+$ sign in the determinant expansion
of a Kasteleyn matrix. There still exist choices of signs satisfying the third
condition for a Kasteleyn matrix, but there is not just one, as in the planar
case, but $2^{2g}$ classes of choices of
signs, yielding $2^{2g}$ non-equivalent Kasteleyn matrices.
From one of them, denoted by $K^{(0)}$, one can construct the other
non-equivalent matrices denoted by $K^{(\eps)}$, $\eps\in\homol{1}$,
by multiplying the
sign of each edge crossing a representative of $\eps$ by $-1$.

In the expansion of the determinant of each of the $2^{2g}$ matrices, there are
terms with different signs.
The sign of a dimer configuration $M$ is determined as follows. Let $M_0$ be a reference dimer configuration
of $\GQ$. Consider the superimposition $M_0\cup M$ of $M_0$ and of the dimer configuration $M$ of $\GQ$. Each 
vertex of $\GQ$ is incident to exactly two edges of
the superimposition, so that we obtain a family of non-intersecting loops and
doubled edges covering all vertices of $\GQ$. Loops of the superimposition also live 
on the surface $\Sigma$ and may have non-trivial homology in $\homol{1}$. 
Given $\eps\in\homol{1}$, the sign of the dimer configuration $M$ in the expansion of the determinant of 
$K^{(\eps)}$ depends on
$\eps$ and the homology class in $\mathbb{Z}/2\mathbb{Z}$ of the 
loops of the superimposition $M_0\cup M$. The determinant of the
Kasteleyn matrix $K^{(\eps)}$ has thus the following form:
\begin{equation*}
  \det K^{(\eps)} = \sum_{\alpha\in\homol{1}} s_{\alpha,\eps} \Zquadri^{(\alpha)},
\end{equation*}
where $\Zquadri^{(\alpha)}$ is the partition function restricted to dimer configurations 
whose superimposition with $M_0$ has homology class $\alpha$, and $s_{\alpha,\eps}$ is a sign
depending only on $\eps$ and $\alpha$.

The remarkable fact is that the linear relations between $\det K^{(\eps)}$ and
$\Zquadri^{(\alpha)}$ can be inverted explicitly and that each $\Zquadri^{(\alpha)}$ can be written as a 
linear combination of the $2^{2g}$ determinants of Kasteleyn matrices
\cite{Russes,GalluccioLoebl,Tesler,CimaReshe1,CimaReshe2}. 
Building on this result, Kenyon \cite{Ke:LocStat} obtains
an explicit expression for the dimer Boltzmann measure on dimer configurations of 
$\Zquadri^{(\alpha)}$ involving the $2^{2g}$ Kasteleyn matrices and their inverses.

Let us return to the purpose of this paper. We are looking for an explicit expression for the law of XOR-polygon
configurations of the double Ising model. By Theorem \ref{thm:main1inside}, this amounts to finding the law of 
dimer configurations $M$ of $\GQ$ such that $\poly_1(M)$ has $0$ homology class. 
The next lemma proves that by choosing the reference dimer configuration $M_0$
appropriately, the homology class of $\poly_1(M)$ is equal to the homology class of the superimposition $M_0\cup M$.

\begin{lem}\label{lem:superposition}
Let $M_0$ be the dimer configuration covering all
internal edges of decorations parallel to dual edges of $G^*$, and let
$M$ be a dimer configuration of $\GQ$. Then, the homology classes of
$M_0\cup
M$ and $\poly_1(M)$ in $\homol{1}$ are equal.
\end{lem}
\begin{proof}
  In order to prove that the homology classes are the same, it suffices to
  prove that they give the same result when computing the intersection form
  against any homology class $\tau\in\homol{1}$.

  Let us fix such a class $\tau\in\homol{1}$. Let $\rep\tau$ be a representative
  of the class $\tau$, realized as a path on $\dual G$. We now show that the
  parity of the number of intersections between $\poly_1(M)$ and $\rep\tau$, is
  equal to the number of intersections between $M\cup M_0$ and $\rep\tau$. All
intersections occur in the interior of dual edge used by $\rep\tau$.

  Fix $\dual e$ a dual edge used by $\rep\tau$. From the mappings above, the edge
  $e$ belongs to $\poly_1(M)$ if and only if the number of interior edges
  parallel to $e$ in the corresponding rhombus covered by dimers in $M$ is odd
  (see Figures \ref{fig:6vertexMapping} and \ref{fig:local_quadri}). Since
  edges of $M_0$ are parallel to $\dual e$, the parity of the number of
  intersections with $\dual e$ will be the same for $\poly_1(M)$ as for $M\cup
  M_0$.  Since this holds for every dual edge belonging to $\rep\tau$, it holds
  for $\rep\tau$. Therefore, the homology classes for $\poly_1(M)$
  and $M\cup M_0$ are the same.\qedhere
\end{proof}

By Lemma \ref{lem:superposition}, the restricted partition function
 $\Zquadri^{(\alpha)}$ is also the partition function restricted to dimer configurations $M$ such
that $\poly_1(M)$ has homology class $\alpha$, this is in particular true for $\alpha=0$. 
As a consequence, the Kasteleyn theory for dimer models defined on graphs embedded in surfaces of genus $g>0$ described above,
yields an explicit expression for the partition function and 
for the probability measure of XOR-polygon configurations of the double Ising model.

Note that classically, 
when studying dimer models one is interested in the full partition function
$  \Zquadri=\sum_{\alpha\in\homol{1}} \Zquadri^{(\alpha)},$
which from the above discussion is a linear combination of the determinants of the $2^{2g}$ Kasteleyn matrices. 
Since considering the restricted model where the homology class of
dimer configurations is fixed is just a matter of considering another linear
combination of determinants, results for the restricted model are readily obtained from those on the full model.
This fact is used in the proof of Theorem \ref{thm:maininfinite}.

%

This discussion of the relation of signs and homology considerations is not
specific to the dimer model on $\GQ$ but applies to any dimer model where the more general Kasteleyn theory
applies, using \emph{full} Kasteleyn matrices with rows and columns indexed by
\emph{all} vertices of the graph, and Pfaffians instead of determinants. We
refer the reader to \cite{CimaReshe1} for an intrinsic geometric
interpretation of coefficients in the linear combination in this general
context. The low-temperature polygon configurations of the Ising model can be
mapped via Fisher's correspondence \cite{Fisher} to a non-bipartite dimer model.
Restricting the homology class of the polygon configurations can also be 
obtained on the dimer side with an appropriate linear combination of Pfaffians
of Kasteleyn matrices. 


\section{The double Ising model at criticality on the whole plane}
\label{sec:dIsing-iso}

After having discussed in much generality the case of finite surfaces of
genus $g$, we now
want to consider the case of infinite planar graphs.
From now on, we restrict ourselves to a special kind of graphs, the so-called
\emph{isoradial graphs}, with specific values of the coupling constants for
the Ising model.

\subsection{Isoradial graphs}\label{subsec:isoradial}

\begin{defi} An \emph{isoradial graph} \cite{Duffin,Mercat:ising,Kenyon3,KeSchlenk} is a planar
graph $G$ together with a proper embedding having the property that every bounded face is inscribed in a circle
of fixed radius, which can be taken equal to 1.
\end{defi}

The regular square, triangular and hexagonal lattices with their standard
embedding are isoradial. A fancier example is given in
Figure~\ref{fig:isoradial} (left).

\begin{figure}[ht]
\begin{center}
\includegraphics[width=12cm]{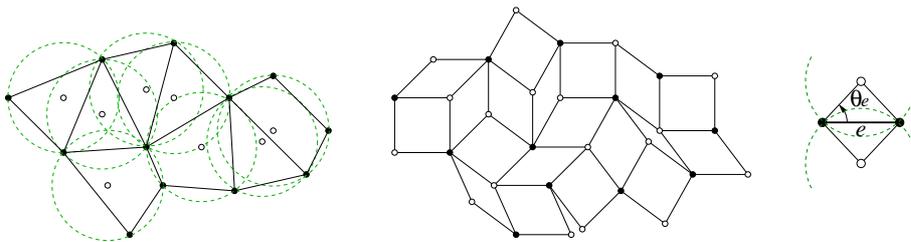}
\caption{Left: an example of an isoradial graph. Middle: the
corresponding rhombus graph. Right: the half-rhombus angle $\theta_e$
associated to an edge $e$.}\label{fig:isoradial}
\end{center}
\end{figure}
The center of the circumscribing circle of a face can be identified
with the corresponding dual vertex, implying that the dual of an isoradial graph
is also isoradial.

The dual of the medial graph $\GM$, called the \emph{diamond graph} of $G$ and
denoted by $\GR$, has as set of vertices the union of those of $G$ and
$\dual G$. There is an edge between $v\in V$ and $\dual f \in V^*$
if and only if $v$ is on the boundary of the face $f$ corresponding to the dual
vertex $\dual f$; see Figure \ref{fig:isoradial} (middle) for an example. Faces
of $\GR$ are rhombi with edge-length 1,
diagonals of which correspond to an edge of $G$ and its dual edge. To each edge
$e$ of $G$ we can therefore associate a geometric angle $\theta_e\in(0,\pi/2)$,
which is the
half-angle of the rhombus containing $e$, measured between $e$ and the edge of
the rhombus; see Figure \ref{fig:isoradial} (right).
 The family $(\theta_e)$ encodes the geometry of the embedding of
the isoradial graph.


\subsection{Statistical mechanics on isoradial graphs}

When defining a statistical mechanical model on an isoradial graph, it is
natural to relate statistical weights to the geometry of the embedding, and thus
to choose the parameters attached to an edge (coupling constants for
Ising, probability to be open for percolation, weight of an edge for dimer
models, conductances for spanning trees or random walk,\dots) to be functions of
the half-angle of that edge. For the Ising model, using
discrete integrability considerations (invariance under star-triangle transformations),
self-duality, the following expression for the interacting constants can be derived, see \cite{Baxter:exactly}:

\begin{equation*}
  J_e=J(\theta_e)=
  \frac{1}{2}\log\left( \frac{1+\sin \theta_e}{\cos \theta_e} \right).
\end{equation*}
These are also known as \emph{Yang--Baxter's coupling constants}.

This expression for $J_e$, when $\theta_e=\frac{\pi}{4}$, $\frac{\pi}{3}$ and
$\frac{\pi}{6}$, coincides with the critical value of the square,
hexagonal and triangular lattices respectively \cite{KramersWannier1,KramersWannier2}. The Ising model with these
coupling constants has been proved to be critical when the isoradial graph is bi-periodic \cite{Li:critical,CimasoniDuminil},
we therefore refer to these values as \emph{critical coupling constants}.

We suppose that the isoradial graph $G$ is infinite, in the sense that the union
of all rhombi
of the diamond graph of $G$ covers the whole plane.

Consider the corresponding bipartite dimer model on the infinite decorated
quadri-tiling graph $\GQ$. Then, the correspondence for the weights described in
Section \ref{sec:52}, Equation \eqref{equ:quadweights} yields in this
particular context:
\begin{equation*}
\forall e\in E,\;  a_e=a(\theta_e)=\cos\theta_e,\quad
\forall e^*\in E^*,\;
b_{e^*}=b(\theta_{e^*})=\cos\theta_{e^*}=\sin\theta_e.
\end{equation*}

Notice that with a particular embedding of the decoration, namely when
external edges have
length 0, and $a$ and $b$ edges form rectangles joining the mid-points of edges
of each rhombus, the quadri-tiling graph $\GQ$ is itself an isoradial
graph with rhombi of edge-length $\frac{1}{2}$, and that weights (up to
a global multiplicative factor of $\frac{1}{2}$) are those introduced by
Kenyon~\cite{Kenyon3} to define critical dimer models on isoradial graphs. 

It turns out that for the Ising and dimer models on infinite isoradial graphs
with critical weights indicated above,
it is possible to construct Gibbs probability measures
\cite{Bea1,BoutillierdeTiliere:iso_gen}, extending the Boltzmann
probability measures in the DLR sense: conditional on the configuration of
the model outside a given bounded region of the
graph, the probability measure of a configuration inside the region is
given by the Boltzmann probability measure defined by the weights above (and the
proper boundary conditions). These measures have the wonderful
property of \emph{locality}: the probability of a local event only depends on
the geometry of a neighborhood of the region where the event takes place, otherwise stated
changing the isoradial graph outside of this region does not affect the
probability.

We can therefore consider the critical Ising model (resp. dimer, in particular
quadri-tilings models) on a general
infinite isoradial graph, as being that particular Gibbs probability measures on
Ising configurations (resp. dimers configurations) of that infinite graph.

We now work with a fixed infinite isoradial graph $G$.
We denote by $\PPising^{\infty}$ the measure on configurations of the
critical Ising model on
$\dual G$. By taking two independent copies of the critical Ising model on
$\dual G$, we
get the Gibbs measure for the critical double Ising model
$\PPdising^{\infty}=\PPising^{\infty}\otimes\PPising^{\infty}$, from which XOR
contours can be constructed, as in the finite case.
We denote by $\PPquadri^{\infty}$ the Gibbs measure on dimer
configurations of the infinite graph $\GQ$.

\subsection{Loops of the critical XOR Ising model on isoradial graphs}

It turns out that the identity in law between polygon configurations of the
critical XOR Ising
model on $G^*$ and those of the corresponding bipartite dimer model on $\GQ$
remains true in the context of infinite isoradial graphs at criticality:

\begin{thm}\label{thm:maininfinite}
Let $G$ be an infinite isoradial graph. The measure induced on polygon
configurations of the critical XOR Ising model on $G^*$, and the measure
induced on
\emph{primal contours} of the corresponding critical bipartite
dimer model on $\GQ$ have
the same law: for any finite subset of edges $\E=\{e_1,\dots,e_n\}$,
\begin{equation}
  \PPdising^{\infty}[\E\subset \XOR] = \PPquadri^{\infty}[\E\subset \poly_1].
  \label{eq:eq_law_contours_critinf}
\end{equation}

\end{thm}

\begin{proof}
Suppose first that the graph $G$ is infinite and bi-periodic, invariant under the translation lattice $\Lambda$.
Then the graph $\GQ$ is also infinite and bi-periodic. 
The infinite volume Gibbs measure $\PPquadri^{\infty}$ on dimer configurations of the bipartite graph $\GQ$ is constructed
in \cite{KOS} as the weak limit of the Boltzmann measures on the natural toroidal exhaustion $\GQ_n=\GQ/n\Lambda$ of 
the infinite bi-periodic graph $\GQ$. The infinite volume Gibbs measure $\PPising^{\infty}$ on 
low temperature Ising polygon configurations of $G$ is constructed in \cite{BoutillierdeTiliere:iso_perio}. 
It uses Fisher's correspondence relating the low temperature expansion of the Ising model and 
the dimer model on a non-bipartite decorated version of the graph $G$. The construction then also 
consists in taking the weak limit of the dimer Boltzmann measures but, since the dimer graph is non-bipartite,
more care is required in the proof of the convergence.

If we apply Theorem~\ref{thm:main1inside} to the specific case of the double 
Ising model on a toroidal, isoradial graph
$G_n^*$ with critical coupling constants, we know that on $G_n$, XOR polygon configurations have the same law as 
  primal polygon configurations of the corresponding bipartite dimer model on $\GQ_n$. The laws involved are the Boltzmann
  measures on configurations having restricted homology. But from Section \ref{sec:quadri_kast}, we know that restricting the homology
  amounts to taking other linear combinations of the Kasteleyn matrices and their inverses. This does neither change the 
  proof of the convergence of the Boltzmann measures, nor the limit. Having equality in law for every $n$ thus
  implies equality in the weak limit.
  
Suppose now that the graph $G$ is infinite but not necessarily periodic. The infinite volume Gibbs measure
$\PPquadri^{\infty}$ of the critical dimer model on the bipartite graph $\GQ$ is constructed in \cite{Bea1}. 
The construction has two main ingredients: the \emph{locality property} meaning that the probability of 
a local event only depends on the geometry of the embedding of $\GQ$ in a bounded domain containing edges involved in the
event, and Theorem 5 of \cite{Bea1} which states that any simply connected subgraph of an infinite rhombus graph
can be embedded in a periodic rhombus graph. The infinite volume Gibbs measure $\PPising^{\infty}$ of the 
low temperature representation of the critical Ising model on $G$ is constructed in \cite{BoutillierdeTiliere:iso_gen}
using the same argument. This implies that one can identify the probability of local events on a non-periodic graph $G$ with
the one of a periodic graph having a fundamental domain coinciding with $G$ on a ball sufficiently large to 
contain a neighborhood of the region of the graph involved in the event.  As a consequence, 
equality in law still holds in the non-periodic case.
\end{proof}

\section{Height function on quadri-tilings}
\label{sec:height}

Dimer configurations of the quadri-tiling graph $\GQ$, like
all bipartite planar dimer models, can be interpreted as random surfaces, via
a \emph{height function}. It is the main ingredient to relate the previous
results connecting XOR loops and dimers with Wilson's conjecture.

\subsection{Definition and properties of the height function}
Let us now recall the definition of height function,
used in \cite{Bea1}.
A dimer configuration $M$ of a planar bipartite graph can be
interpreted as a unit flow $\alpha_M$, flowing by 1 along each matched edge of
$M$, from the white vertex to the black one. It is a function on edges
having divergence $+1$ at each white vertex and $-1$ at each black vertex.
Subtracting from $\alpha_M$ another flow with the same divergence
at every vertex, yields a divergence-free flow, whose dual is the differential
of a function on faces of this graph.

There is a natural candidate for this unit reference flow: since in a dimer
configuration there is exactly one dimer incident to every vertex, the
sum over all edges incident to any given vertex of the probability that this
edge is covered by a dimer, is equal to 1. This means that the 
flow $\alpha_0$, flowing by $\PPquadri^{\infty}(e)$\footnote{The graph $\GQ$ is
isoradial and infinite, and the
weights for the  quadri-tilings are critical. So in this particular context, we
know \cite{Kenyon3} that the probability of an edge is given by $\theta/\pi$,
where
$\theta$ is the half-angle of the rhombus containing that edge.}
from the white vertex to the black one along each edge $e$ of the graph, is a
flow with divergence $+1$ (resp. $-1$) at every white (resp. black) vertex.

The height function $h$ on quadri-tilings is defined as follows.
For every dimer configuration $M$ of $\GQ$, $h^M$ is
a function on faces of $\GQ$, such that for every pair of neighboring faces $f$
and $f'$ of $\GQ$ sharing an edge $e$, with the additional property that when
traversing $e$ from $f$ to $f'$, the black vertex of $e$ is on the left:
\begin{equation*}
  h^M(f')-h^M(f)= \alpha_M(e)-\alpha_0(e).
\end{equation*}

When faces $f$ and $f'$ are not incident, choose a path
$f=f_0,f_1,\ldots,f_n=f'$ in the dual graph joining $f$ and $f'$, then:
\begin{equation*}
h^M(f')-h^M(f)=\sum_{i=0}^{n-1}(h^M(f_{i+1})-h^M(f_i)).
\end{equation*}
This definition is consistent, \emph{i.e.}, independent of the choice of path from
$f$ to $f'$, because the flow $\alpha_M-\alpha_0$ is
divergence free; it determines $h^M$ up to a
global additive constant, which can be fixed by saying that the
height at a particular given face of $\GQ$ is $0$.
Faces of $\GQ$ are split into three distinct subsets, those corresponding to:
vertices of $G$, vertices of the dual $G^*$ and edges of $G$ (or $G^*$).
We suppose for the sake of definiteness that the face where the height is fixed at $0$ corresponds to 
some particular vertex of $G$.

Denote by $h_V^M$ (resp. $h_{\dual V}^M$) the restriction of $h^M$ to vertices
of
$G$ (resp. to vertices of $\dual G$).

The next lemma describes possible height changes between pairs of vertices of
the primal (resp. dual) graph, incident in the primal (resp. dual) graph. To
simplify the picture, we consider primal and dual vertices to be around a
rhombus of the diamond graph; see Figure~\ref{fig:quadri5}.

\begin{figure}[ht]
\begin{center}
\includegraphics[width=14cm]{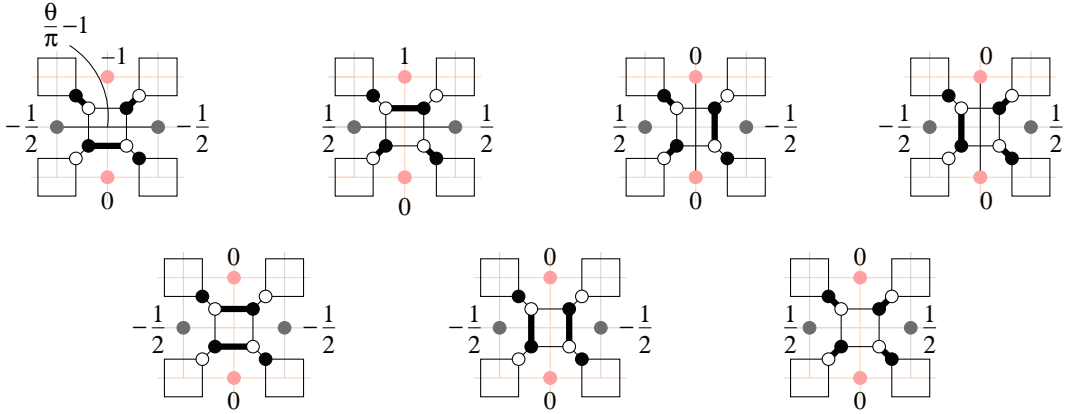} 
\caption{Height changes for the dimer model in a rhombus of the diamond
graph.}\label{fig:quadri5}
\end{center}
 
\end{figure}

\begin{lem}$\ $
\begin{itemize}
  \item The function $h^M_V$ (resp. $h^M_{\dual V}$) takes values in $\ZZ$ (resp.
    $\ZZ+\frac{1}{2}$).
  \item The increment of $h^M_V$ (resp. $h^{M}_{\dual V}$) between two
    neighboring vertices of $G$ (resp. of $\dual G$) is $-1$, $0$, or $1$.
  \item The increment of $h^M_V$ (resp. $h^M_{\dual V}$) is non-zero if and
only if the two vertices are separated
    by an edge of $\poly_2(M)$ (resp. $\poly_1(M)$).
\end{itemize}
\end{lem}

\begin{proof}
  Let $e_1$, $e_2$ be two interior edges of $\GQ$, parallel to an
edge of $G$, as in Figure \ref{fig:quadri5}. Then, the reference flow $\alpha_0$
has the
same value but opposite direction on these two edges. As
  a consequence, using the definition of the height function,
$$
h^M(v_2^*)-h^M(v_1^*)=\II_{e_1}(M)-\PPquadri^{\infty}(e_1)
-\II_{e_2}(M)+\PPquadri^{\infty}(e_1)=\II_{e_1}(M)-\II_{e_2}(M).
$$
A similar expression holds for $h^M(v_2)-h^M(v_1)$.
This proves that the increment of $h^M$ between two neighboring vertices of
$G$ (resp. $G^*$) is equal to $-1$, $0$, or $1$. Because of our convention for
the base point,
this implies that $h^M$ takes integer values on $G$. To see that $h^M$ takes
half-integer values on $G^*$, one just has to notice that the reference flow
$\alpha_0$ separating two vertices $v$ (on $G$) and $\dual v$ (on
$\dual G$) which are neighbors on $\GR$ is
$\frac{\pi/2}{\pi}=\frac{1}{2}$ since the corresponding rhombus in the isoradial
graph $\GQ$ is flat.
\end{proof}

\begin{rem}
  Note that another choice of reference unit flow is the one coming
  from the reference dimer configuration $M_0$, where a white-to-black unit is
  flowing along all interior edges parallel to edges of $G^*$. This
  produces a random height function whose restriction to
  vertices of $G$, resp. of $\dual G$, coincides with $h_V$, resp. $h_{\dual V}$
  (up to an additive constant).
\end{rem}

\begin{rem}
  The height function $h^M$ on $V\cup \dual V$ can be defined directly from
  the 6-vertex dimer configuration, using the representation in terms of
  orientations depicted in Figure~\ref{fig:6vconfigs}. Since the number of
  incoming and outgoing edges is the same at each vertex, the set of edges in
  the 6-vertex configuration can be partitioned into oriented contours. These
  contours are the level lines of the restriction of $h^M$ to $V\cup \dual V$
  separating two successive half-integer values, and can thus be used to
  reconstruct $h^M$.
\end{rem}

The \emph{level lines} of $h_V$ (resp. $h_{\dual V}$) are the set of closed
contours on
$\dual G$ (on resp $G$) separating clusters of vertices of $G$ (resp. of $\dual
G$) where $h_V$ (resp. $h_{\dual V}$) takes the same value. 

Returning to the definition of the pair of polygon configurations $\poly(M)$
assigned to a quadri-tiling $M$, we immediately obtain the following:

\begin{lem}\label{lem:22}
Let $M$ be a dimer configuration of $\GQ$, then 
level lines of $h_V^M$, respectively $h_{V^*}^M$, exactly correspond to the
polygon configuration $\poly_1(M)$, respectively $\poly_2(M)$.
\end{lem}

Note that due to the fact that $\poly_1(M)$ and $\poly_2(M)$ do not cross, the
increments of $h^M$ along two diagonals of a rhombus cannot be both non-zero. As
a consequence, on contour lines of $h_V^M$, $h_{\dual V}^M$ is constant.

Combining Lemma \ref{lem:22} with Theorem \ref{thm:maininfinite} stating that
monochromatic polygon configurations of the XOR
Ising model have the same distribution as primal polygon configurations
of dimer configurations of $\GQ$, we obtain one of the main theorems of this
paper:

\begin{thm}\label{thm:xor-level}
Monochromatic polygon configurations of the critical XOR-Ising model have the
same distribution as level lines of the restriction to primal vertices of the
height function of
dimer configurations of $\GQ$. 
\end{thm}

\subsection{Wilson's conjecture}
\label{sec:wilsonconj}
In \cite{WilsonXOR}, Wilson presented extensive numerical simulations on loops
of the critical XOR Ising model on the honeycomb lattice, on the base of which
he conjectured the following:

\begin{conj}[Wilson \cite{WilsonXOR}]
  The scaling limit of the family of loops of the critical XOR Ising model are
  the level lines of the Gaussian free field corresponding to levels that are
  odd multiples of $\frac{\sqrt{\pi}}{2}$.
\end{conj}

The Gaussian free field is a wild object: it is a random \emph{generalized}
function, and not a function, and as such, there is no direct way to define what
its level lines are. The level lines of the Gaussian free field are understood
here as the scaling limit when the mesh goes to zero of the level lines of the
discrete Gaussian free field on a triangulation of the domain, which separate
domains where the field is above or below a certain
level~\cite{SchrammSheffield}.

The level lines of the Gaussian free field corresponding to levels that are odd
multiples of  $\lambda=\sqrt{\frac{\pi}{8}}$ form a
$\mathrm{CLE}_4$ \cite{SchrammSheffield,MillerSheffield}. The contour lines of the XOR Ising models are thus conjectured
to have the same limiting behavior as the $\mathrm{CLE}_4$, except that there
are a factor of $\sqrt{2}$ times fewer loops in the XOR Ising picture. This conjecture is in
agreement with predictions of conformal field
theory~\cite{IkhlefRajabpour,Picco}.

Theorem~\ref{thm:xor-level} can be interpreted as a proof of a version of Wilson's
conjecture in a discrete setting, before passing to the scaling limit, and
brings some elements for the complete proof of
this conjecture. In particular, it explains the link with the Gaussian free
field and the factor $\sqrt{2}$, as we will now show.

For $\varepsilon>0$, denote by $\GQ_\varepsilon$ the
embedding of $\GQ$ in the plane where rhombi of $\GQ$ have side length
$\varepsilon$. For every dual vertex $v$ in $\dual \GQ$, define
$v^\varepsilon$ the vertex in ${\GQ_\varepsilon}^*$ corresponding to the dual
vertex $v$.

The random height function $h$ can be interpreted on $\GQ_\varepsilon$ as a
\emph{random distribution} \cite{GelfandVil}, \emph{i.e.}, a continuous random linear form on the set
$\mathcal{C}^\infty_{0,c}(\mathbb{R}^2)$ of compactly
supported smooth, zero mean functions, denoted by $H^\varepsilon$:
for every $\varphi\in\mathcal{C}^\infty_{0,c}(\RR^2)$, 
\begin{equation*}
  H^\varepsilon(\varphi)= \sum_{v\in\dual \GQ} \mathrm{area}(v^\varepsilon) h(v)
  \varphi(v^\varepsilon),
\end{equation*}
where $\mathrm{area}(v^\varepsilon)=\varepsilon^2 \mathrm{area}(v)$ is the area
of the face of $\GQ_\varepsilon$ associated to $v^\varepsilon$.

In \cite{Bea2}, the second author proved the following convergence result
for the height function of the dimer model on $\GQ$:

\begin{thm}[\cite{Bea2}]
  As $\varepsilon$ goes to 0, the height function on the critical
  quadri-tilings, as a random distribution, converges in law to
  $\frac{1}{\sqrt{\pi}}$ times the Gaussian free field.
\end{thm}

The result also holds for the restriction of $h$ to $G$ (resp. to $\dual G$) as
soon as $\mathrm{area}(v)$ is replaced by the area of the corresponding face of
$G$ (resp. $\dual G$).

As contour lines of the restriction of $h$ to $G$ separate integer values,
they can be understood as discrete level lines corresponding to half-integer
values. Therefore, it is natural to expect that these contour lines converge to
the contour lines of the limiting object, \emph{i.e.}, to level lines for the
Gaussian free
field with levels $(k+\frac{1}{2})\sqrt{\pi}$, $k\in\ZZ$, which would prove
Wilson's conjecture. Unfortunately, the result for the convergence result of the
height function to the Gaussian free field is too weak to ensure convergence of
contour lines.

The convergence result in the paper \cite{Bea2} applies not only to critical
quadri-tilings, but to all
bipartite planar dimer models on isoradial graphs with critical weights. It is
conjectured that the family of loops obtained by superimposing two independent
critical dimer configurations converges to CLE$_4$. This is supported by the
fact that each of the dimer configurations can be described by a height
function, converging in the scaling limit to $1/\sqrt{\pi}$ times the Gaussian
Free Field, the two fields being independent. Dimer loops are the half integer
level lines of the difference, which by independence converges (in a weak sense) to
$\sqrt{2/\pi}$ times the Gaussian free field, and it is known that level
lines $(k+1/2)\sqrt{\frac{\pi}{2}}$ of the Gaussian free field are a CLE$_4$.

Therefore, the factor $\sqrt{2}$ in Wilson's conjecture corresponds to the fact
that contours in the XOR Ising model have to do with contour lines of
only one dimer height function, as opposed to two for dimer loops.

\appendix

\section{Some elements of homology theory on surfaces}\label{app:A}

Here are some general facts about homology theory on surfaces which are
useful in the context of this paper. More details can be found in the references
\cite{Fulton,Maunder,Massey}.
We consider $\Sigma$ to be a compact,
orientable surface of genus $g$ with boundary $\partial\Sigma$ consisting of $p$
components. The boundary may be empty, in which case $p=0$.

We are interested in the \emph{first} homology group $H_1$, and
in the case where the target abelian group is $\ZZ/2\ZZ$. The
other non-trivial homology groups $H_0$ and $H_2$ are isomorphic to $\ZZ/2\ZZ$,
when $\Sigma$ is connected.

\subsection{$1$-chains and first homology group}\label{app:A1}

A \emph{$1$-chain} is a formal linear combination of
1-dimensional submanifolds of $\Sigma$.
The coefficients here will be taken to be in $\ZZ/2\ZZ$.
The space of $1$-chains with coefficients in $\ZZ/2\ZZ$ is a
$\ZZ/2\ZZ$-vector space.
Note that since the target group is $\ZZ/2\ZZ$, we do not need
to care about orientations of 1-chains, and the sum is the same as the
difference.
If two $1$-chains are the sums of pairwise disjoint submanifolds
\begin{equation*}
  \gamma = \sum_{c\in A}c, \quad
  \gamma' = \sum_{c\in A'}c,
\end{equation*}
with $c\neq c' \Rightarrow c\cap c'=\emptyset$, then
\begin{equation*}
  \gamma+\gamma'=\sum_{c\in A\cup A'}c,
\end{equation*}
so that we can think morally of the addition as the union for disjoint 1-chains.

The boundary of a 1-chain is the formal linear combination of the end points
of the 1-dimensional submanifolds it consists of.
A 1-chain is a
\emph{cycle} if its boundary is empty.
An equivalence relation on the space of cycles is defined as follows:
two cycles $\gamma$ and $\gamma'$ are equivalent if their sum is the boundary
of a 2-dimensional submanifold of
$\Sigma$.
Note that a connected component of the boundary of a 2-dimensional submanifold
of $\Sigma$ may be not itself the whole boundary of a submanifold, and as such
may not be equivalent to the empty chain for the relation above.
The \emph{first homology
group} $\homol{1}$ is the set of equivalence classes for this relation.
It has a structure of $\ZZ/2\ZZ$-vector space inherited from the one of
the space of $1$-chains.

\subsection{Relative homology}\label{app:A2}

If $A$ is a closed subset of $\Sigma$, then one can also consider
\emph{homology relative to $A$}. As above, one defines a notion
of cycle and boundary, relative to $A$ this time: a \emph{relative cycle} is a
$1$-chain whose
boundary is in $A$. A \emph{relative boundary} is a 1-chain in $\Sigma$ for
which there exists a $1$-chain in $A$, such that the sum of the two is a
boundary in $\Sigma$. Associated to this concept is an equivalence relation
on relative cycles: two relative cycles are equivalent if their sum is a relative
boundary.
Then, the \emph{first homology group relative to $A$},
denoted by $H_1(\Sigma, A; \ZZ/2\ZZ)$, is the $\ZZ/2\ZZ$-vector space generated by
relative cycles, quotiented by this equivalence relation.

A particular example of interest is when $A=\partial\Sigma$. Representatives of
equivalence classes of $\relhomol{1}$ are finite unions of cycles and paths
attached to components of the boundary.

Note that when $\partial\Sigma$ is empty, the homology of $\Sigma$
relative to its boundary coincides with the usual homology.

\subsection{Explicit bases of homology}\label{app:A3}

The two homology groups $\homol{1}$ and $\relhomol{1}$ turn out to have the same
dimension
\begin{equation*}
  N=
  \begin{cases}
    2g& \text{if $p=0$ or 1,}\\
    2g+p-1& \text{otherwise.}\\
  \end{cases}
\end{equation*}
For each of the groups, a basis can be explicitly given.
Label the handles of $\Sigma$ from 1 to $g$, and the
$p$ components of the boundary from $C_0$ to $C_{p-1}$. For
$\homol{1}$, choose $N$ cycles
$({\rep\lambda}_i)_{i=1}^N$ on $\Sigma$, as follows:
\begin{itemize}
  \item for $i\in\{1,\dots, g\}$, take ${\rep\lambda}_{2i-1}$ and
    ${\rep\lambda}_{2i}$ to be winding around the $i$-th handle in two
transverse directions,
  \item for $i\in\{1,\dots, N-2g\}$, take ${\rep\lambda}_{2g+i}$ to be winding around
    $C_i$, without crossing ${\rep\lambda}_1,\dots,{\rep\lambda}_{2g}$.
\end{itemize}
Denote by $\lambda_i$ the homology class of ${\rep\lambda}_i$.
Then, the collection $(\lambda_i)_{i=1}^N$ is a basis of $\homol{1}$:
any $1$-chain on $\Sigma$ has the same homology class as a sum of
$\lambda_i$'s. The first homology group $\homol{1}$ is isomorphic 
to $(\ZZ/2\ZZ)^{N}$: for every
$i\in\{1,\ldots,N\}$, the basis element $\lambda_i$ is mapped to the basis
element of $(\ZZ/2\ZZ)^{N}$ consisting of $0$'s and a 1 at position $i$.

For $\relhomol{1}$, choose $N$ cycles $(\rep\gamma_i)_{i=1}^N$ on $\Sigma$, as
follows:
\begin{itemize}
  \item for $i\in\{1,\dots,g\}$, take ${\rep\gamma}_{2i-1}={\rep\lambda}_{2i}$, and
    ${\rep\gamma}_{2i}={\rep\lambda}_{2i-1}$,
  \item for $i\in\{1,\dots,N-2g\}$, take ${\rep \gamma}_{2g+i}$ to be a path from
    $C_0$ to $C_i$.
\end{itemize}
Denote by $\gamma_i$ the relative homology class of ${\rep\gamma}_i$. Then, the
collection $(\gamma_i)_{i=1}^N$ is a basis for $\relhomol{1}$. The group
$\relhomol{1}$ is also
isomorphic to $(\ZZ/2\ZZ)^{N}$.
The two bases $(\lambda_i)_{i=1}^N$ and $(\gamma_i)_{i=1}^N$ are dual to each
other as explained in Appendix \ref{app:intersection}.

\subsection{Representatives of homology classes on graphs}\label{app:graphs}

Consider a cellular decomposition of the surface $\Sigma$ by a graph $G_{\Sigma}$, as in
Section~\ref{sec1}. The embedding of $G_{\Sigma}$ on $\Sigma$ defines a notion
of dual graph for $G_{\Sigma}$, denoted by $\dual G_{\Sigma}$. Then
representatives of any homology class
of $\homol{1}$ (resp. any relative homology class of $\relhomol{1}$) can be
realized as combinatorial paths on $\dual G$ (resp. $G$). Figure
\ref{fig:homol_basis} provides an example of representatives of the bases
$(\lambda_i)_{i=1}^N$ and $(\gamma_i)_{i=1}^N$ defined in Section \ref{app:A3}.

\begin{figure}[ht]
\begin{center}
\includegraphics[width=10cm]{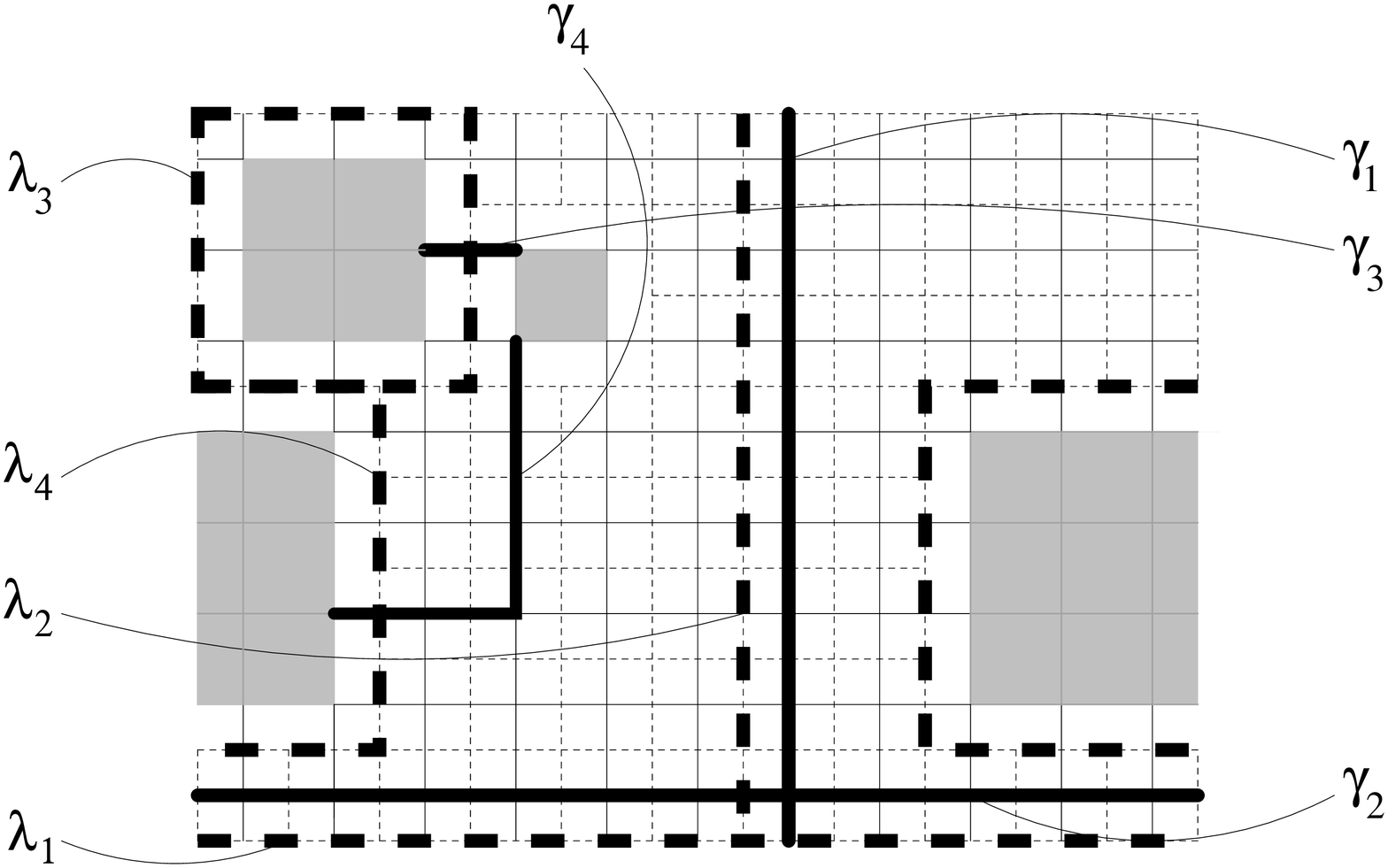}
\caption{Representatives of a basis $(\lambda_i)_{i=1}^4$ of $\homol{1}$
(dotted lines), and of a basis $(\gamma_i)_{i=1}^4$ of $\relhomol{1}$ (plain
lines).}
\label{fig:homol_basis}
\end{center} 
\end{figure}

\subsection{Intersection form}\label{app:intersection}

There is a natural pairing between $\homol{1}$ and $\relhomol{1}$, called the
\emph{intersection form}:
\begin{equation*}
  (\cdot |\cdot): \homol{1}\times\relhomol{1} \longrightarrow \ZZ/2\ZZ,
\end{equation*}
defined as follows. Let $\tau\in \homol{1}$ and $\eps\in \relhomol{1}$ be two
homology classes. Take representatives $\rep\tau$ and $\rep\eps$ for
two classes $\tau$ and $\eps$ respectively. Then $(\tau|\eps)$ is defined as the
parity of the number of intersections of
${\rep\tau}$ and ${\rep\eps}$. This definition does not depend on the choice of
representatives.

In the explicit bases of $\homol{1}$ and $\relhomol{1}$ chosen in Appendix~\ref{app:A3},
the matrix of the intersection form is the identity. The pairing
is thus non-degenerate and defines an isomorphism between $\homol{1}$ and
$\relhomol{1}$. This is an explicit realization of the Poincar\'e--Lefschetz
duality; see Theorem~5.4.13 and Corollary~5.2.12 in \cite{Maunder}.

\subsection{Inclusion, excision and morphisms for homology}\label{app:A6}

Suppose that there exists a larger surface $\tilde{\Sigma}$ containing $\Sigma$.
Then a $1$-chain in
$\Sigma$ is in particular a chain in $\tilde{\Sigma}$, and a boundary in
$\Sigma$
is in particular a boundary in $\tilde{\Sigma}$. This implies that the inclusion
$\Sigma\subset \tilde{\Sigma}$ induces a morphism
\begin{equation*}
  \pi_{\tilde\Sigma,\Sigma}: \homol{1}\longrightarrow \homol[\tilde{\Sigma}]{1}.
\end{equation*}

The inclusion also induces morphisms for relative homology groups: if the
subset $A\subset\Sigma$  is included in a subset $B\subset\tilde{\Sigma}$, then
any relative chain (resp. cycle) in $\Sigma$ relative to $A$ is in particular a
relative chain (resp. cycle) in $\tilde{\Sigma}$ relative to $B$ (just by
forgetting what is in $B\setminus A$). Therefore, this induces a morphism
\begin{equation*}
  H_1(\Sigma, A; \ZZ/2\ZZ) \longrightarrow H_1(\tilde{\Sigma},B; \ZZ/2\ZZ),
\end{equation*}
giving the homology class in $\tilde{\Sigma}$ relative to $B$ of
the restriction to $\tilde{\Sigma}\setminus B$ of any representative of an
element of $H_1(\Sigma, A; \ZZ/2\ZZ)$. In the special case when
$\Sigma=\tilde{\Sigma}$ and $A$ is empty, we get the application
$\iota_{\tilde\Sigma,B}$:
\begin{equation*}
  \iota_{\tilde\Sigma,B}=H_1(\tilde\Sigma; \ZZ/2\ZZ) \longrightarrow
  H_1(\tilde{\Sigma},B; \ZZ/2\ZZ).
\end{equation*}

Moreover, the \emph{excision theorem} (\cite{Maunder}, Theorem~8.2.1) states that if we cut out an open set $U$ from
both $\tilde{\Sigma}$ and $B$, the relative homology groups
$H_1(\tilde{\Sigma}, B; \ZZ/2\ZZ)$ and $H_1(\tilde{\Sigma}\setminus U,
B\setminus U; \ZZ/2\ZZ)$ are isomorphic.
In particular, when $U= \Sigma^{c}=\tilde{\Sigma}\setminus\Sigma$ and
$B=\overline{U}$, then the
excision theorem states that $H_1(\tilde{\Sigma}, \Sigma^{c}; \ZZ/2\ZZ)$ and
$\relhomol{1}$ are isomorphic. Let $e_{\tilde{\Sigma},\Sigma}$ the isomorphism
from the former space to latter.
The composition $\Pi_{\tilde\Sigma,\Sigma}=e_{\tilde{\Sigma},\Sigma}\circ
\iota_{\tilde\Sigma,B}$ defines a morphism from $H_1(\tilde\Sigma; \ZZ/2\ZZ)$ to
$\relhomol{1}$.

To construct a representative of $\Pi_{\tilde\Sigma,\Sigma}(\epsilon)$ for a
homology class $\epsilon\in H_1(\tilde\Sigma; \ZZ/2\ZZ)$, consider
$\rep{\epsilon}$ a cycle representing $\epsilon$ in $\tilde{\Sigma}$. A
representative of $\Pi_{\tilde\Sigma,\Sigma}(\epsilon)$ is then simply obtained
by taking the intersection of $\rep{\epsilon}$ with $\Sigma$, which is a
relative 1-chain of $\Sigma$ relative to its boundary $\partial\Sigma$.

\bibliographystyle{alpha}
\bibliography{survey}

\end{document}